\documentclass[a4paper,10pt]{amsart}
\usepackage[utf8]{inputenc}
\usepackage{amssymb,amsmath,dsfont,tikz-cd}
\usepackage[mathscr]{eucal}
\usepackage{stmaryrd}
\usepackage{hyperref}
\usepackage[all]{xy}
\usepackage{upgreek}
\usepackage{colonequals}
\usepackage{tikz-cd}
\usepackage{color}
\usepackage{verbatim}
\usepackage[shortlabels]{enumitem}

\newcommand{\blue}[1]{{\color{blue} #1}}

\usepackage[normalem]{ulem}

\renewcommand{\theta}{\uptheta}
\renewcommand{\alpha}{\upalpha}
\renewcommand{\beta}{\upbeta}
\renewcommand{\gamma}{\upgamma}
\renewcommand{\delta}{\updelta}
\renewcommand{\zeta}{\upzeta}
\renewcommand{\pi}{\uppi\hspace{0.05em}}
\renewcommand{\rho}{\uprho}
\renewcommand{\xi}{\upxi}
\renewcommand{\chi}{\upchi}
\renewcommand{\sigma}{\upsigma}
\renewcommand{\Lambda}{\Uplambda}
\renewcommand{\Gamma}{\Upgamma}
\renewcommand{\phi}{\upphi}
\renewcommand{\nu}{\upnu}
\renewcommand{\tau}{\uptau}
\renewcommand{\mu}{\upmu}
\renewcommand{\eta}{\upeta}

\newtheorem{theorem}{Theorem}[section]

\newtheorem{definition}[theorem]{Definition}

\newtheorem{proposition}[theorem]{Proposition}
\newtheorem{lemma}[theorem]{Lemma}

\newtheorem{corollary}[theorem]{Corollary}
\newtheorem{question}[theorem]{Question}

\makeatletter
\newenvironment{proofof}[1]{\par
  \pushQED{\qed}%
  \normalfont \topsep6\p@\@plus6\p@\relax
  \trivlist
  \item[\hskip\labelsep
        \scshape
    Proof of #1\@addpunct{.}]\ignorespaces
}{%
  \popQED\endtrivlist\@endpefalse
}
\makeatother

\DeclareMathOperator{\supp}{supp}

\DeclareMathOperator{\fram}{fr}
\DeclareMathOperator{\fr}{fr}
\DeclareMathOperator{\gr}{gr}
\DeclareMathOperator{\udim}{\underline{\dim}}

\renewcommand{\AA}{\mathbb{A}}

\DeclareMathOperator{\cnilp}{c-nilp}

\DeclareMathOperator{\BGl}{BGL}
\newcommand{\CC}{\mathbb{C}}
\newcommand{\OO}{\mathscr{O}}

\DeclareMathOperator{\ncHilb}{ncHilb}
\DeclareMathOperator{\EXP}{EXP}
\DeclareMathOperator{\LOG}{LOG}
\DeclareMathOperator{\Fr}{Fr}
\DeclareMathOperator{\lngth}{lt}

\DeclareMathOperator{\BPSmo}{BPS}
\newcommand{\BPSA}[1]{\BPSmo^*_{#1}}
\DeclareMathOperator{\Sta}{Sta}
\DeclareMathOperator{\Grot}{K_0}

\DeclareMathOperator{\vir}{vir}
\newcommand{\Comm}{\mathbf{C}}
\newcommand{\StComm}{\mathfrak{C}}

\newcommand{\muhat}{\hat{\mu}}

\newcommand{\TTr}{\mathcal{T}r}
\newcommand{\TTTr}{\mathfrak{T}r}

\newcommand{\BPS}{\mathcal{BPS}}
\newcommand{\GG}{\mathbb{G}}

\DeclareMathOperator{\rat}{rat}
\DeclareMathOperator{\PV}{\mathbf{P}}
\DeclareMathOperator{\PS}{\mathfrak{P}}

\newcommand{\LL}{\mathbb{L}}
\newcommand{\LLL}{\mathfrak{L}}
\newcommand{\NN}{\mathbb{N}}

\newcommand{\QQ}{\mathbb{Q}}

\newcommand{\WW}{\mathcal{T}r(W)}
\newcommand{\WWW}{\mathfrak{Tr}(W)}
\newcommand{\XX}{\mathbb{X}}
\newcommand{\X}{\mathcal{X}}

\newcommand{\Yy}{\mathcal{Y}}

\newcommand{\ZZ}{\mathbb{Z}}
\newcommand{\Mst}{\mathfrak{M}}

\newcommand{\Msp}{\mathcal{M}}

\newcommand{\ICS}{\overline{\mathcal{IC}}}

\newcommand{\phim}[1]{\phi^{\mon}_{#1}}

\newcommand{\MM}{\mathscr{M}}
\newcommand{\MMM}[1]{\MM^{\mon}_{#1}}

\DeclareMathOperator{\moLQ}{Q}
\newcommand{\LQ}[1]{\moLQ^{(#1)}}\DeclareMathOperator{\Mat}{Mat}

\DeclareMathOperator{\crit}{crit}

\DeclareMathOperator{\Hom}{Hom}
\DeclareMathOperator{\KK}{K}
\newcommand{\AAA}[1]{\AA_{\CC}^{#1}}
\newcommand{\Khat}[1]{\hat{\KK}_0^{#1}}

\DeclareMathOperator{\mon}{mon}
\DeclareMathOperator{\stab}{st}

\DeclareMathOperator{\MMHM}{MMHM}

\DeclareMathOperator{\Gr}{Gr}

\DeclareMathOperator{\Rep}{Rep}

\DeclareMathOperator{\Id}{Id}

\DeclareMathOperator{\Perv}{Perv}
\DeclareMathOperator{\MHM}{MHM}
\DeclareMathOperator{\MHS}{MHS}
\DeclareMathOperator{\MMHS}{MMHS}

\DeclareMathOperator{\Sym}{Sym}
\DeclareMathOperator{\SSym}{\mathfrak{S}}

\DeclareMathOperator{\red}{red}
\DeclareMathOperator{\Var}{Var}

\DeclareMathOperator{\Spec}{Spec}
\DeclareMathOperator{\SpecSym}{SpecSym}
\DeclareMathOperator{\Gl}{GL}
\DeclareMathOperator{\id}{id}

\DeclareMathOperator{\Tr}{Tr}

\DeclareMathOperator{\pt}{pt}

\DeclareMathOperator{\forg}{forg}

\DeclareMathOperator{\Ob}{Ob}
\DeclareMathOperator{\Ho}{\mathcal{H}}
\DeclareMathOperator{\HO}{H}
\DeclareMathOperator{\HP}{HP}

\newcommand{\Dblf}{\mathcal{D}^{\leq, \mathrm{lf}}}
\newcommand{\Db}{\mathcal{D}^{b}}
\newcommand{\Dsg}{\mathcal{D}_{\sg}}
\DeclareMathOperator{\sg}{sg}
\newcommand{\Dub}{\mathcal{D}}

\let\bb=\mathbb
\let\wt=\widetilde
\let\llb=\llbracket
\let\rrb=\rrbracket
\let\ol=\overline
\let\mc=\mathcal

\title[Deformed dimensional reduction]{Deformed dimensional reduction}
\author{Ben Davison and Tudor P\u adurariu}

\begin{document}
 
\begin{abstract} 
Since its first use by Behrend, Bryan, and Szendr\H{o}i in the computation of motivic Donaldson--Thomas (DT) invariants of $\AA_{\CC}^3$,
dimensional reduction has proved to be an important tool in motivic and cohomological DT theory. 
Inspired by a conjecture of Cazzaniga, Morrison, Pym, and Szendr\H{o}i on motivic DT invariants, work of Dobrovolska, Ginzburg, and Travkin on exponential sums, and work of Orlov and Hirano on equivalences of categories of singularities, we generalize the dimensional reduction theorem in motivic and cohomological DT theory and use it to prove versions of the Cazzaniga--Morrison--Pym--Szendr\H{o}i conjecture in these settings.
\end{abstract}

\maketitle
\setcounter{tocdepth}{1}
\tableofcontents

\section{Introduction}
This paper concerns generalizations of dimensional reduction in Donaldson--Thomas (DT) theory, which has proven to be an indispensable tool in calculating DT invariants in various versions of the theory:
motivic \cite{BBS}, cohomological \cite{Dav16b}, and K-theoretic \cite{P}.  

\subsection{Dimensional reduction}
\label{DimRedSec}
We start by recalling a motivic version of the dimensional reduction theorem, which is a slight variant of the one proved in \cite{BBS}.  Let $X$ be a complex algebraic variety, and let $\GG_m$ act on 
\begin{equation}
\label{set2}    
\overline{X}=X\times\AAA{m}
\end{equation}
by $z\cdot (x,t)=(x,zt)$.  Assume that $g\in \Gamma(\overline{X})$ is a regular degree one function, so that we may write 
\begin{equation}
    \label{set1}
g=\sum_{1\leq j\leq m} g_jt_j
\end{equation}
where $g_j$ are functions pulled back from $X$ and $t_1,\ldots,t_m$ are coordinates on $\AAA{m}$.  Let $\overline{Z}\subset\overline{X}$ be the reduced vanishing locus of the functions $g_1,\ldots,g_m$. The theorem states that 
\begin{equation}
\label{mdr}
\int[\phi_g]=\LL^{-\frac{\dim(\overline{X})}{2}}[\overline{Z}]\in\Khat{}(\Var/\pt)
\end{equation}
where $\int[\phi_g]$ is the absolute motivic vanishing cycle defined by Denef and Loeser \cite{DL01}.  Their definition lies in a ring of $\muhat$-equivariant motives, whereas the identity (\ref{mdr}) takes place in a naive Grothendieck ring of motives with no monodromy; part of the statement of the theorem is that the monodromy on the left hand side of (\ref{mdr}) is in fact trivial, so that this makes sense. 
\smallbreak
There is a cohomological version of the theorem as well.  Let 
\[
\phim{g}\colon\Dub(\MHM(\overline{X}))\rightarrow\Dub(\MMHM(\overline{X}))
\]
be the vanishing cycle functor.  Note the extra ``M'' appearing in the target category; this stands for monodromy, and is again accounted for by the monodromy automorphism on the vanishing cycles.  By construction, for any mixed Hodge module $\mathcal{F}$, $\phim{g}\mathcal{F}$ is supported on $\overline{X}_0\colonequals g^{-1}(0)$, and moreover there is a natural transformation $\phim{g}\mathcal{F}\rightarrow \mathcal{F}\lvert_{\overline{X}_0}$.  Since $\overline{Z}\subset \overline{X}_0$, we can restrict further to obtain the natural transformation
\begin{equation}
\label{fv}
\phim{g}\mathcal{F}\rightarrow\mathcal{F}\lvert_{\overline{Z}}.
\end{equation}
Denote by $i\colon\overline{Z}\hookrightarrow \overline{X}$ the inclusion.  We can alternatively obtain (\ref{fv}) by applying $\phim{g}$ to the natural transformation 
\begin{equation}
\label{nti}
\id\rightarrow i_*i^*,
\end{equation}
since the vanishing cycle functor commutes with proper maps, is the identity functor for the zero function, and $gi=0$.  The cohomological dimensional reduction theorem \cite[Thm.A.1]{DAV} states that the natural map
\begin{equation}
\label{urdr}
\pi_!\phim{g}(\id\rightarrow i_*i^*)\pi^*
\end{equation}
is an isomorphism for $\mathcal{G}\in \MHM(X)$.  Just as in the motivic version of the theorem, the target has trivial monodromy, since $\pi_!\phim{g}i_*i^*\pi^*\cong\pi_!i_*i^*\pi^*$.  

Let $r\colon S\hookrightarrow X$ be the inclusion of a subvariety, let $\overline{S}=\pi^{-1}(S)$, and let $\tau\colon S\rightarrow\pt$ be the structure morphism.  Then a consequence of the theorem is the statement that 
\[
\tau_!r^*\pi_!\phim{g}(\id\rightarrow i_*i^*)\pi^*\mathbb{Q}_X
\]
is an isomorphism, and so there is an isomorphism of (monodromic) mixed Hodge structures
\begin{equation}
    \label{ADR}
\HO_c(\overline{S},\phim{g}\QQ_{\overline{X}})\cong\HO_c(\overline{Z}\cap \overline{S},\mathbb{Q})
\end{equation}
where the right hand side has trivial monodromy.  This is the special case that is used most often.
\smallbreak
\subsection{Cohomological deformed dimensional reduction}
The starting point of this paper is the question of whether we can generalize in the following way.  Assume instead that for $g\in\Gamma(\ol{X})$ we can write 
\[
g=g_0+\sum_{1\leq j\leq m}g_j t_j
\]
where $g_0,\ldots,g_m$ are again pulled back from functions on $X$. Write
\begin{equation}
\label{gredDef}
g^{\red}\colonequals  g_0\lvert_{\overline{Z}}
\end{equation}
where $\overline{Z}$ is the vanishing locus of $g_1,\ldots,g_m$ as before.  The natural transformation (\ref{urdr}) is still defined, and we obtain from it a natural transformation
\begin{equation}
\label{dri}
\pi_!\phim{g}\pi^*\rightarrow\pi_!i_*\phim{g^{\red}}i^*\pi^*,
\end{equation}
bearing in mind that $g\lvert_{\overline{Z}}=g^{\red}$.  The purpose of this paper is to answer the following
\begin{question}
\label{mainq}
Is \eqref{dri} an isomorphism?  
\end{question}
By applying the natural transformation \eqref{dri} to $\QQ_X$ and taking total compactly supported hypercohomology, we obtain as before a homomorphism of monodromic mixed Hodge structures
\[
\HO_c(\overline{X},\phim{g}\QQ_{\overline{X}})\rightarrow\HO_c(\overline{Z},\phim{g^{\red}}\mathbb{Q}_{\overline{Z}})
\]
and we may ask if it is an isomorphism.  Obviously a positive answer to Question \ref{mainq} implies a positive answer to this question.  

There are several situations in which the answer to Question \ref{mainq} is yes:
\begin{enumerate}[(A)]
\item
If $g_0=0$, then $\phim{g^{\red}}\cong\id$ and (\ref{dri}) becomes naturally isomorphic to (\ref{urdr}) which is an isomorphism by the usual dimensional reduction theorem.
\item
Let $X=T\times\AAA{p}$, with $g_1,\ldots,g_m$ pulled back from $T$, and $g_0=\sum_{1\leq j\leq p}h_jt'_j$, with $t'_j$ coordinates on $\AAA{p}$ and $h_1,\ldots,h_p$ pulled back from $T$, and assume furthermore that $\mathcal{G}\cong\pi_{T}^*\mathcal{G}'$ for some $\mathcal{G}'\in\Ob(\Dub(\MHM(T)))$.  Then by two applications of the dimensional reduction theorem, (\ref{dri}) is an isomorphism when applied to $\mathcal{G}$.
\item
Let $X=X_1\times X_2$ with $g_1,\ldots,g_m$ pulled back from $X_1$ and $g_0$ pulled back from $X_2$.  Then one may prove, using the Thom--Sebastiani isomorphism and the usual dimensional reduction isomorphism, that (\ref{dri}) is an isomorphism.
\item
\label{bifur}
Let $X=\AAA{1}=\mathrm{Spec}(\CC[x])$, and set $n=1$.  Consider the regular functions $g_1(x)=x^a$ and $g_0(x)=x^b$ with $b\geq a$.  We have that $\overline{X}=\AAA{2}$ and $\pi$ is the projection map
\begin{align*}
    \pi\colon &\AAA{2}\rightarrow\AAA{1}\\
    &(x,t)\mapsto x
\end{align*}
and 
\begin{align*}
    g=&x^a(t+x^{b-a}).
\end{align*}
We have $Z=Z(x)=0\in\AAA{1}$ and $\overline{Z}=\AAA{1}$, embedded in $\AAA{2}$ as the $y$-axis.  The singular locus of $g$ is contained in this copy of $\AAA{1}$, which is the fiber of $\pi$ over zero, and so  \begin{equation}
    \nonumber
\pi_!\phim{g}\QQ_{\overline{X}}=\HO_c(\AAA{2},\phim{g}\QQ_{\AAA{2}})\otimes \QQ_{\{0\}}.  
\end{equation}
Since $g^{\red}=0$, it follows that $\phim{g^{\red}}\QQ_{\overline{Z}}=\QQ_{\overline{Z}}$ and thus
\begin{equation}
    \nonumber
\pi_!\phim{g^{\red}}\QQ_{\overline{Z}}=\HO_c(\AAA{1},\QQ)\otimes\QQ_{\{0\}}.
\end{equation}
Finally, rewriting $g=x^at'$, where we have changed coordinates by setting $t'=t+x^{b-a}$, we deduce that 
\[
\HO_c(\AAA{2},\phim{g}\QQ_{\AAA{2}})\cong \HO_c(\AAA{1},\QQ)
\]
via the usual dimensional reduction theorem, and so $\pi_!\phim{g}\QQ_{\overline{X}}$ and $\pi_!\phim{g^{\red}}\QQ_{\overline{Z}}$ are isomorphic.  Furthermore, one can show that \eqref{dri} is indeed an isomorphism.

\end{enumerate}

As encouraging as these observations are, it turns out that it is not hard to cook up examples for which the answer to Question \ref{mainq} is no.  For instance, modify example (D) from above so that now $a$ and $b$ satisfy $a> b$.  Then again $Z=Z(x)$, and $g^{\red}=0$, so that the right hand side of (\ref{dri}) is again given by $\HO_c(\AAA{1},\QQ)\otimes\QQ_{\{0\}}$, and so is nontrivial, but with trivial monodromy.  We now write
\[
g=x^b(tx^{a-b}+1).
\]
If $b=1$, zero is not a critical value, and so the left hand side of (\ref{dri}) is zero when applied to the constant sheaf $\QQ_{X}$. If $b>1$, one can check that
\[
\phim{g}\QQ_{\AAA{2}}\cong \HO_c(\AAA{1},\phim{x^b})\otimes\QQ_{\overline{Z}}
\]
where $\overline{Z}\cong\AAA{1}$.  In particular, $\pi_!\phim{g}\QQ_{\AAA{2}}$ has \textit{nontrivial} monodromy.  Putting these observations together, we see that there can be no isomorphism (\ref{dri}) in case $a>b$.


Considering the well behaved and badly behaved variants of (D) above, we see that the dimensional reduction morphism is an isomorphism if and only if there is a non-negative weighting of $x$ and $t$ making $g$ a quasihomogeneous function with positive weight. This brings us to our main theorem, which will be proved in \S \ref{MainThmSec}. We first state it in a particular case, which will make the comparison with the usual dimensional reduction theorem considered in \S \ref{DimRedSec} transparent:

\begin{theorem}
\label{MainThmCor}
Let $\overline{X}=X\times\AAA{n}\xrightarrow{g}\AAA{1}$ be a $\GG_m$-equivariant function, where $\GG_m$ acts trivially on $X$, with non-negative weights on $\AAA{n}$, and with positive weight on $\AAA{1}$. Assume furthermore that there is a $\GG_m$-equivariant decomposition $\AAA{n}=\AAA{m}\times\AAA{n-m}$ and that we can write 
\[
g=g_0+\sum_{1\leq j\leq m}g_jt_j
\]
with the functions $g_0,\ldots,g_m$ pulled back from $X\times\AAA{n-m}$.  Let $\pi\colon \overline{X}\rightarrow X$ be the natural projection.  Then the dimensional reduction natural transformation
\begin{equation}
\label{prime3}
\pi_!\phim{g}\pi^*\rightarrow\pi_!i_*\phim{g^{\red}}i^*\pi^*
\end{equation}
is an isomorphism of functors, and for $S$ and $\ol{S}$ as in Theorem \ref{MainThm} there is a natural isomorphism
\begin{equation*}
\HO_c(\overline{S},\phim{g}\QQ_{\overline{X}})\cong \HO_c(\overline{Z}\cap \overline{S},\phim{g^{\red}}\QQ_{\overline{Z}}).
\end{equation*}
\end{theorem}
Note that Theorem \ref{MainThmCor} implies the usual dimensional reduction theorem.  Indeed, if $g$ is as in \eqref{set1} and $\ol{X}$ is as in \eqref{set2}, we give $\AAA{m}$ the scaling action, set $n=m$, and the conditions of the theorem are satisfied.

Before we state the general form of our main theorem, we recall the following basic construction: let 
\begin{equation}
    \label{sescoh}
0\rightarrow \mathscr{V}'\xrightarrow{} \mathscr{V}\rightarrow\mathscr{V}''\rightarrow 0
\end{equation}
be a short exact sequence of locally free sheaves on a variety $X$.  We denote by 
\begin{align*}
\rho''\colon&\SpecSym(\mathscr{V})\rightarrow \SpecSym(\mathscr{V}')\\\rho'\colon&\SpecSym(\mathscr{V}')\rightarrow X
\end{align*}
the induced maps of varieties.  Locally we can split the short exact sequence and write \[\mathscr{V}_U\cong\mathscr{O}_U^{\oplus n}\cong \mathscr{O}_U^{\oplus m}\oplus \mathscr{O}_U^{\oplus (n-m)},\,\,\,\,\mathscr{V}'_U\cong \mathscr{O}_U^{\oplus (n-m)}.\]
This induces an isomorphism \[\Sym_{\mathscr{O}_U}(\mathscr{V}_U)\cong \Sym_{\mathscr{O}_U}(\mathscr{V}'_U)\otimes \Sym_{\mathscr{O}_U}(\mathscr{V}''_U)\]
and a decomposition of $\Sym_{\mathscr{O}_U}(\mathscr{V}_U)$ by degree in $\Sym_{\mathscr{O}_U}(\mathscr{V}''_U)$.  This decomposition depends on the splitting, but the degree filtration does not, so that the degree filtration is well-defined.  Moreover, if we denote by $\mathscr{D}_{\mathscr{V}',\leq 1}\subset \Sym_{\mathscr{O}_X}(\mathscr{V})$ the $\mathscr{O}_{X}$-submodule of degree one functions, there is a short exact sequence
\[
0\rightarrow \Sym_{\mathscr{O}_X}(\mathscr{V}')\rightarrow \mathscr{D}_{\mathscr{V}',\leq 1}\xrightarrow{\xi}\rho'^*\mathscr{V}''\rightarrow 0. 
\]

\begin{theorem}
\label{MainThm}
Let $\mathscr{V}$ be a $\GG_m$-equivariant locally free sheaf on a variety $X$, where $\GG_m$ acts with non-negative weights on $\mathscr{V}$, and equip $\ol{X}=\SpecSym(\mathscr{V})$ with the $\GG_m$-action from $\mathscr{V}$.  Let $g\in\Gamma(\ol{X})^{\chi}$ be a $\GG_m$-semi-invariant function on $\ol{X}$ for a positive $\GG_m$-character $\chi$.  Assume furthermore that there is a short exact sequence of locally free $\GG_m$-equivariant sheaves \eqref{sescoh} and that $g\in\mathscr{D}_{\mathscr{V}',\leq 1}$.

Define $i:\ol{Z}\hookrightarrow \ol{X}$ to be the inclusion of $\rho''^{-1}(Z(\xi(g)))$.  Define $g^{\red}\in\Gamma(\ol{Z})$ as in (\ref{gredDef}), and let $\pi\colon \ol{X}\rightarrow X$ be the natural projection. 
  Then the dimensional reduction natural transformation
\begin{equation}
\label{DRI}
\pi_!\phim{g}\pi^*\rightarrow\pi_!i_*\phim{g^{\red}}i^*\pi^*
\end{equation}
is an isomorphism of functors.  In particular, for $S\subset X$ a subvariety and $\overline{S}=\pi^{-1}(S)$, there is a natural isomorphism of monodromic mixed Hodge structures
\begin{equation}\label{prime2}
\HO_c(\overline{S},\phim{g}\QQ_{\overline{X}})\cong \HO_c(\overline{Z}\cap \overline{S},\phim{g^{\red}}\QQ_{\overline{Z}}).
\end{equation}
\end{theorem}

If $\mathscr{V}=\mathscr{O}_X^{\oplus n}$ and $\mathscr{V}'=\mathscr{O}_X^{\oplus (n-m)}$ for some $m,n$, and we have a split short exact sequence
\[
0\rightarrow \mathscr{V}'\rightarrow \mathscr{V}\rightarrow \OO_X^{\oplus m}\rightarrow 0,
\]
then the statement of Theorem \ref{MainThm} simplifies to the statement of Theorem \ref{MainThmCor}.  Indeed, since it is enough to prove locally that \eqref{DRI} is an isomorphism, it will be enough to prove that \eqref{prime3} is; see \S \ref{MainThmSec} for the details.

\subsection{Motivic deformed dimensional reduction}
Just as in the case of ordinary dimensional reduction, there is a motivic version of Theorem \ref{MainThm}; we prove this version as Theorem \ref{MotDDR}, and give its exact statement here.

We assume the setup of Theorem \ref{MainThm}, and moreover that $X$ is a smooth connected variety.
Assume that $g$ has weight $d$.  For $t\in \CC$, define $\overline{Z}_t\colonequals (g^{\red})^{-1}(t)$.  We endow $\overline{Z}_0$ with the trivial $\muhat$-action, and endow $\overline{Z}_1$ with the $\muhat$-action factoring through the natural $\mu_d$-action given by restricting the $\GG_m$-action on $\overline{Z}$.  The statement\footnote{Actually, the statement takes place inside the isomorphic ring $\Khat{\mon}(\Var/X)$, which turns out to be much easier to work with.  See \S \ref{MotBackground} for details.} of Theorem \ref{MotDDR} is that there is an equality
\begin{equation}
\label{MotDDRst}    
\pi_{X,!}[\phi^{\mon}_g]=\LL^{-\frac{\dim(\overline{X})}{2}}\left([\overline{Z}_0\xrightarrow{} X]-[\overline{Z}_1\xrightarrow{} X]\right)\in \Khat{\muhat}(\Var/X).
\end{equation}

Note that, in contrast to \eqref{mdr}, the monodromy here may be nontrivial; our main application, detailed in the following section, is such a case.  

Theorem \ref{MotDDR} turns out to be a consequence of a theorem of Nicaise and Payne \cite{NP17}, stating that for certain $\GG_m$-equivariant functions $g$, the ``naive'' motivic nearby fiber $[g^{-1}(1)]$ and the motivic nearby fiber defined by Denef and Loeser agree.

\subsection{Motivation from Donaldson--Thomas theory}\label{motiv}
One of the motivations for searching for a generalization of the dimensional reduction isomorphism was a conjecture of Cazzaniga, Morrison, Pym, and Szendr\H{o}i, regarding the motivic Donaldson--Thomas invariants of the quiver $\LQ{3}$ with three loops $a,b,c$, and with the homogeneous deformed Weyl potential $W_3=a[b,c]+c^3$.  The definition of these invariants is recalled in \S \ref{DTdefs}.  They conjectured that for all $n\in\mathbb{Z}_{\geq 1}$, there is an equality 
\[
\Omega_{\LQ{3},W_3,n}=\LL^{1/2}(1-[\mu_3]).
\]

For $d\geq 2$, we can consider the DT theory for the (quasihomogeneous) deformed Weyl potential $W_d=a[b,c]+c^d$ via deformed dimensional reduction of the potential $\wt{W}=a[b,c]$ (in fact in this paper we will treat all quasihomgeneous deformations of $\wt{W}$ in the two variables $b,c$).  The DT theory of the undeformed pair $(\LQ{3},\wt{W})$ is very well understood: Behrend, Bryan and Szendr\H{o}i proved that the motivic DT invariants for $(\LQ{3},\wt{W})$ are given by $\LL^{3/2}$ for all $n$. 

Later, in \cite{Dav16b}, the cohomological DT invariants of $(\LQ{3},\wt{W})$ were calculated, along with their relative versions.
The cohomological version of the DT theory of a quiver with potential $(Q,W)$ is recalled in \S \ref{coDTdefs} below.  The central objects of study in this theory are the \textit{BPS sheaves} $\BPS_{Q,W,\gamma}$ on $\Msp_\gamma(Q)$, the coarse moduli space of $\gamma$-dimensional $Q$-representations; the refined DT invariants for the pair $(Q,W)$ are obtained by taking weight polynomials of the compactly supported hypercohomology of these sheaves \cite{DaMe15b}.  Using purity, it was shown in \cite[Sec.5]{Dav16b} that 
\begin{equation}
    \label{3loopBPS}
\BPS_{\LQ{3},\wt{W},n}=\Delta_{n,*}\QQ_{\AAA{3}}\otimes \LLL^{-3/2}
\end{equation}
where 
\begin{align*}
\Delta_n\colon&\AAA{3}\rightarrow \Msp_n(\LQ{3})\\
&(x,y,z)\mapsto(x\cdot\Id_{n\times n}, y\cdot\Id_{n\times n}, z\cdot\Id_{n\times n})
\end{align*}
and $\LLL^{1/2}$ is a half Tate twist.

Firstly, using the motivic version of the deformed dimensional reduction theorem, we verify the Cazzaniga--Morrison--Pym--Szendr\H{o}i conjecture (the case $d=3$ corresponds to their original conjecture):

\begin{theorem}\label{motdef}
For all $d\geq 2$, the motivic DT invariants of $(\LQ{3},W_d)$ are
\[
\Omega_{\LQ{3},W_d,n}=\LL^{1/2}(1-[\mu_d])
\]
for all $n\in\ZZ_{\geq 1}$.
\end{theorem}

In the case $n=1$ this theorem follows essentially from the definitions.  The case $d=3$ and $n=2$ was proved by Le Bruyn \cite{LLB16}; by hand it is already a heavy computation.

Secondly, using purity again, along with this motivic version of the Cazzaniga--Morrison--Pym--Szendr\H{o}i conjecture, we prove a cohomological refinement of their conjecture at the end of \S \ref{DefWeylCoh}.  
\begin{theorem}\label{cohdef}
Let $\Msp_n(\LQ{3})$ be the coarse moduli space of $n$-dimensional $\LQ{3}$-representations. We identify $\AAA{2}$ with the subspace in $\AAA{3}$ given by the $xy$-plane.  There is an isomorphism in $\MMHM(\Msp_n(\LQ{3}))$
\begin{equation}
\label{CMPSR}
\BPS_{\LQ{3},W_d,n}\cong \Delta_{n,*}\QQ_{\AAA{2}}\otimes \HO_c(\AAA{1},\phim{t^d}\QQ)\otimes\LLL^{-3/2},
\end{equation}
as well as an isomorphism of cohomologically graded monodromic mixed Hodge structures
\begin{equation}
\label{CMPSA}    
\BPSA{\LQ{3},W_d,n}\cong \HO_c(\AAA{1},\phim{t^d}\QQ)\otimes\LLL^{1/2}.
\end{equation}
\end{theorem}

The second statement follows from the first, since in general $\BPSA{Q,W,\gamma}$ is defined\footnote{Note that this is the dual of the definition of BPS cohomology from \cite[Thm.A]{DaMe15b}, since we take the compactly supported cohomology of the (Verdier self-dual) BPS monodromic mixed Hodge module instead of the cohomology.} to be the compactly supported hypercohomology of $\BPS_{Q,W,\gamma}$.  

A striking feature of this cohomological version of the Cazzaniga--Morrison--Pym--Szendr\H{o}i conjecture is that it can be restated by saying that there is an isomorphism
\begin{align}
\label{DGT}\nonumber
\BPS_{\LQ{3},W_d,n}\cong&\phim{\TTr(c^d)}\BPS_{\LQ{3},\wt{W},n}\\:=&\phim{\TTr(c^d)}\phim{\TTr(\wt{W})}\ICS_{\Msp_n(\LQ{3})}(\QQ)\otimes \LLL^{-\dim(\Msp_n(\LQ{3}))},
\end{align}
where the second equality comes from the definition of $\BPS_{\LQ{3},\wt{W},n}$.  This observation provided one of the main motivations for suspecting and then proving that some deformation of the dimensional reduction theorem exists, and this approach lies behind our second proof of the cohomological version of their conjecture in \S \ref{ppBPS_sec}, using Theorem \ref{MainThm}.  

\subsection{BPS invariants for preprojective algebras with potentials}
In fact Theorem \ref{cohdef} turns out to follow from a special case of a general theorem on BPS invariants, which is inspired by comparing with \cite{DGT16}.  In \cite{DGT16}, working in finite characteristic, Dobrovolska, Ginzburg, and Travkin obtain formulae for the analogues of characteristic functions of vanishing cycle sheaves on the infinitesimal inertia stacks of the stacks of representations of quivers.  They show that these sums are given by plethystic exponentials of Frobenius traces of cohomology of certain coadjoint orbits, twisted by potentials.  By work of Hausel, Letellier and Villegas \cite{HLV}, it is precisely the cohomologies of these coadjoint orbits that recover DT invariants in the complex setting.  On the other hand, in \cite{HLV} there are no potentials.  This is because their setting is ``dimensionally reduced'', i.e. corresponds to the \textit{target} of isomorphisms such as \eqref{ADR}.  It is thus intriguing that potentials still play a key role in the results of \cite{DGT16}.

We recall here that the preprojective algebra for a quiver $Q$ is defined by
\begin{equation}
    \label{preProjDef}
\Pi_Q\colonequals \CC \ol{Q}/\langle \sum_{a\in Q_1} [a,a^*]\rangle,
\end{equation}
where $\ol{Q}$ is the double quiver associated to $Q$.
The formulae obtained by Dobrovolska--Ginzburg--Travkin in \cite{DGT16} suggest that the vanishing cycle cohomology of potentials on the stack of representations of the preprojective algebra could itself arise as BPS cohomology.  This was one of the motivating suggestions for pursuing deformed dimensional reduction.  Our final theorem makes this precise.  

To state it, we first recall the \textit{tripled} quiver $\wt{Q}$, which is obtained from $\ol{Q}$ by adding a loop $\omega_i$ at every vertex $i\in Q_0$.  The tripled quiver carries the canonical cubic potential 
\[
\wt{W}=\sum_{i\in Q_0}\omega_i\sum_{a\in Q_1}[a,a^*].
\]
There is a forgetful map 
\[\varpi\colon \Msp(\wt{Q})\rightarrow \Msp(\overline{Q})\]
along with an $\AAA{1}$-family of sections 
\[l\colon\Msp(\overline{Q})\times\AAA{1}\rightarrow\Msp(\wt{Q})\]
given by setting the action of all of the $\omega_i$ to be multiplication by $z\in\AAA{1}$.  
Furthermore, by the results of \cite{Dav16b}, we can write \[\BPS_{\wt{Q},\wt{W},\gamma}= l_*(\BPS_{\Pi_Q,\gamma}\boxtimes\QQ_{\AAA{1}})\otimes\LLL^{-1/2}\]
for certain monodromic mixed Hodge modules on $\Msp_{\gamma}(\ol{Q})$ supported on the locus of $\Pi_Q$-representations.
\begin{theorem}
\label{DGT_compare}
Let $Q$ be a finite quiver, let $W'\in \CC\ol{Q}/[\CC\ol{Q},\CC\ol{Q}]$ be a potential, and assume that $\wt{W}+W'$ is quasihomogeneous.  Define
\begin{align*}
\mathcal{G}_{\gamma}:=\HO_c(\Msp_{\gamma}(\ol{Q}),\phim{\Tr(W')}\BPS_{\Pi_Q,\gamma})\otimes\LLL.
\end{align*}
Then there are isomorphisms
\begin{align*}
\Sym\left( \bigoplus_{\gamma\in\NN^{Q_0}\setminus \{0\}} \mathcal{G}_{\gamma}\otimes\HO_c(\pt/\CC^*)\right)\cong&\bigoplus_{\gamma\in\NN^{Q_0}}\HO_c(\Mst_{\gamma}(\Pi_Q),\phim{\TTTr(W')}\QQ_{\Mst_{\gamma}(\Pi_Q)})\otimes \LLL^{\chi_Q(\gamma,\gamma)}\\\cong&
\bigoplus_{\gamma\in\NN^{Q_0}}\HO_c(\Mst_{\gamma}(\wt{Q}),\phim{\TTTr(\wt{W}+W')}\QQ_{\Mst_{\gamma}(\wt{Q})})\otimes \LLL^{\chi_{\wt{Q}}(\gamma,\gamma)/2}.
\end{align*}
\end{theorem}
Consider the special case in which $Q=\LQ{1}$ is the Jordan quiver, so that $\wt{Q}\cong\LQ{3}$. Then using the explicit description of the BPS sheaves \eqref{3loopBPS}, Theorem \ref{DGT_compare} gives rise to \eqref{CMPSA} via Corollary \ref{GenCor}.

\subsection{Categorical dimensional reduction.}
For a regular function on a smooth variety $f\colon X\to\AAA{1}$, Efimov \cite{e} showed that the ($\mathbb{Z}/2\mathbb{Z}$-periodic) vanishing cycle cohomology 
$\HO(X_0,\varphi_f\mathbb{Q})$ is categorified by the category of singularities $\Dsg(X_0)$.
A categorification of the dimensional reduction theorem was proved by Isik \cite{I}.
The statement of Theorem \ref{MainThm} is inspired by work of Orlov \cite{O2} and Hirano \cite{hi} who prove a categorification of Theorem \ref{MainThm} when $\overline{Z}$ is smooth. 

It would be interesting to look for a categorification (or a K-theoretic version) of Theorem \ref{MainThm} without assuming that $\overline{Z}$ is smooth. Such a result will have applications in K-theoretic DT theory \cite{P}, and is to be the subject of future work.

\subsection{Acknowledgements}
The paper emerged from discussions with Victor Ginzburg on the results of \cite{DGT16} and the possibility of a generalization of dimensional reduction; BD is immensely grateful to him for hosting him at the University of Chicago in November 2016, and to Brent Pym for explaining some of the key features of his conjecture with Cazzaniga, Morrison, and Szendr\H{o}i. TP thanks Davesh Maulik for numerous discussions about the results of this paper.

During the writing of the paper, BD was supported by
the starter grant ``Categorified Donaldson-Thomas theory'' No.~759967 of the European Research Council, which also enabled TP to visit Edinburgh.  BD was also supported by a Royal Society 
university research fellowship.

\subsection{Conventions and notations}
All the schemes and stacks considered in this paper are defined over $\CC$. We use the notation $\mathbb{G}_m=\Gl_1$.  As an algebraic variety, we identify $\GG_m$ with the complement of $0$ in $\AAA{1}$.  As a group, we identify it with $\CC^*$, and use $\CC^*$ and $\GG_m$ interchangeably.

All functors are assumed to be derived. We set
\[
\NN=\{m\in\ZZ\colon\medskip m\geq 0\}. 
\]

The motive $\LL$ is defined as $[\AAA{1}]$, and its square root $\LL^{1/2}$ is defined in Equation (\ref{squareroot}).

The notations related to quivers and their moduli of representations are introduced in \S \ref{DTdefs}.  Throughout, we denote by $\LQ{r}$ a quiver with one vertex, labelled $1$, and $r$ loops.

For a complex variety $X$, we denote by $\MMHM(X)$ the category of monodromic mixed Hodge modules on $X$, see \S \ref{mixedHodge} for their definitions. We use the notation $\MMHS=\MMHM(\pt)$ for the category of monodromic mixed Hodge structures. We use $\psi$ and $\varphi$ for nearby and vanishing cycle functors for constructible sheaves, and $\uppsi$ and $\phi$ for the nearby and vanishing cycle functors for mixed Hodge modules, see \S \ref{vancycles} for their definitions. The definitions of the full and half Tate twists (monodromic mixed Hodge structures) $\LLL$ and $\LLL^{1/2}$ are given in \eqref{defLLL} and \eqref{defLLLh}, respectively.

For a complex variety $X$, we denote by $\QQ_X$ the constant constructible sheaf on $X$ with stalks given by $\QQ$.  We denote by the same symbol the natural upgrade of this sheaf to an object in the derived category of mixed Hodge modules.

\section{Motivic deformed dimensional reduction}
\subsection{Background and definitions}
\label{MotBackground}
Let $G$ be an algebraic group.  For $Y$ a $G$-equivariant variety, we denote by $\Khat{G}(\Var/Y)$ the free Abelian group generated by symbols
\begin{equation}
    \label{effective_class}
[X\xrightarrow{f} Y]
\end{equation}
where $f$ is a morphism of $G$-equivariant varieties and $X$ is reduced, with two types of relations:
\begin{enumerate}
    \item 
The cut and paste relations
\[
[X\xrightarrow{f} Y]=[U\xrightarrow{f\lvert_U}Y]+[Z\xrightarrow{f\lvert_Z}Y]
\]
for $U\subset X$ an open subvariety with closed complement $Z$.  
\item
The relation
\[
[V\rightarrow X\xrightarrow{f} Y]=[X\times\AAA{n}\xrightarrow{f\circ \pi_X} Y]
\]
if $V\rightarrow X$ is the projection from the total space of a rank $n$ $G$-equivariant vector bundle.
\end{enumerate}

We call sums of elements as in \eqref{effective_class} \textit{effective}.  General elements in $\Khat{G}(\Var/Y)$ and variants of this ring can be written as $A-B$ where $A$ and $B$ are effective.  

Let $d\in\ZZ_{\geq 1}$, and give $Y\times \AAA{1}$ the $\GG_m$-action $z\cdot (y,t)=(y,z^dt)$, thus defining the group $\Khat{\GG_m,d}(\Var/Y\times \AAA{1})$.

We denote by $\mu_d\subset\GG_m$ the group of $d$th roots of unity.                  We define the groups $\Khat{\GG_m,d}(\Var/Y\times \GG_m)$ and $\Khat{\mu_d}(\Var/Y)$ similarly, giving $Y$ the trivial $\mu_d$-action in the second case.  There is an isomorphism of groups
\begin{align*}
&\Khat{\mu_d}(\Var/Y)\rightarrow \Khat{\GG_m,d}(\Var/Y\times \GG_m)\\
&[X\xrightarrow{f}Y]\mapsto {}-[X\times_{\mu_d}\GG_m\xrightarrow{(x,z)\mapsto (f(x),z^d)}Y\times \GG_m].
\end{align*}
This isomorphism sends a variety with a $\mu_d$-action to a variety over $\GG_m$, locally constant in the \'etale topology, with monodromy given by the $\mu_d$-action.  The $\GG_m$-equivariant inclusion $\GG_m\hookrightarrow \AAA{1}$ induces an inclusion of Abelian groups
\[
\iota_d\colon \Khat{\GG_m,d}(\Var/Y\times \GG_m)\rightarrow\Khat{\GG_m,d}(\Var/Y\times \AA_{\CC}^1)
\]
via composition.  There is an inclusion of groups
\begin{align*}
\nu_d^*\colon &\Khat{}(\Var/Y)\rightarrow \Khat{\GG_m,d}(\Var/Y\times \AA_{\CC}^1)\\
&[X\xrightarrow{f} Y]\mapsto[X\times \AA_{\CC}^1\xrightarrow{f\times\id} Y\times \AA_{\CC}^1]
\end{align*}
where the $\GG_m$-action on the target is defined by $z\cdot (y,z')=(y,z^dz')$.  The group $\Khat{\GG_m,d}(\Var/\AA_{\CC}^1)$ carries a ring structure defined by
\begin{align}
\label{LambdaOps1}
[X_1\xrightarrow{f_1}\AA_{\CC}^1]\cdot [X_2\xrightarrow{f_2}\AA_{\CC}^1]\colonequals  [X_1\times X_2\xrightarrow{+\circ (f_1\times f_2)} \AA_{\CC}^1]
\end{align}
and $\Khat{\GG_m,d}(\Var/Y)$ carries a $\Khat{\GG_m,d}(\Var/\pt)$-module structure defined in the same way.  If $f_1=f_2$, there is an extra $\SSym_2$-action on the right hand side of (\ref{LambdaOps1}): we define
\begin{align}
\label{extBox}
[X\xrightarrow{f_1}\AA_{\CC}^1]^{\boxdot\:n}=[X^n\xrightarrow{+\circ f^{\times n}}\AA_{\CC}^1]\in\Khat{\SSym_n\times\GG_m,d}(\Var/\AA_\CC^1)
\end{align}
where the symmetric group $\SSym_n$ acts trivially on $\AA_{\CC}^1$.  For $n\geq 0$, define operations $\sigma^n$ on effective classes via
\begin{align}
\label{LambdaOps2}
&\sigma^n[X\xrightarrow{f} \AA_{\CC}^1]\colonequals  [\Sym^n(X)\xrightarrow{+\circ \Sym^n(f)}\AA_{\CC}^1]=\pi_{(n)}\left([X\xrightarrow{f}\AA_\CC^1]^{\boxdot\:n}\right)
\end{align}
where we slightly abuse notation and denote by the same symbol the morphisms
\begin{align*}
&+\colon\mathbb{A}^1_{\CC}\times\mathbb{A}^1_{\CC}\rightarrow \AA^{1}_{\CC}\\
&+\colon\Sym^n(\AA_{\CC}^1)\rightarrow\AA_{\CC}^1
\end{align*}
provided by addition.  The group morphism $\pi_{(n)}$ in (\ref{LambdaOps2}) is defined by
\begin{align}
\label{pidef}
\pi_{(n)}\colon \Khat{\SSym_n\times\GG_m,d}(\Var/\AA_\CC^1)&\rightarrow \Khat{\GG_m,d}(\Var/\AA_\CC^1)\\
[X\rightarrow \AAA{1}]&\mapsto[X/\SSym_n\rightarrow \AAA{1}].\nonumber
\end{align}
The operations $\sigma^n$ can be extended uniquely to all classes in $\Khat{\GG_m,d}(\Var/\AA_{\CC}^1)$ via the relation
\begin{equation}
    \label{sigma_extend}
\sigma^n(A+B)=\sum_{i=0}^n\sigma^i(A)\sigma^{n-i}(B).
\end{equation}

For the proof that $\sigma^n$ can be extended in this way see \cite{GZLMH}.  Obviously \eqref{LambdaOps2} and \eqref{sigma_extend} then determine the operations $\sigma^n$ uniquely.

The subgroup $\nu_d^*(\Khat{}(\Var/\pt))\subset \Khat{\GG_m,d}(\Var/\AA_{\CC}^1)$ is a $\lambda$-ideal, and so the quotient
\[
\Khat{\GG_m,d}(\Var/\AA_{\CC}^1)/\nu_d^*\left(\Khat{}(\Var/\pt)\right)\cong\Khat{\mu_d}(\Var/\pt)
\]
acquires the structure of a pre-$\lambda$-ring; for details on $\lambda$- and pre-$\lambda$-rings, see \cite[Sec.3, Rmk.3.4]{DM11}.  For $d'\vert d$, there is a natural inclusion of pre-$\lambda$-rings
\[
\Khat{\mu_{d'}}(\Var/\pt)\hookrightarrow \Khat{\mu_d}(\Var/\pt)
\]
given by the morphism 
\begin{align*}
&\mu_d\rightarrow \mu_{d'}\\
&\zeta\mapsto \zeta^{d/d'}
\end{align*}
and we define by $\Khat{\muhat}(\Var/\pt)$ the pre-$\lambda$-ring obtained as the limit of these inclusions.  In particular, there is an inclusion 
\[
\Khat{}(\Var/\pt)=\Khat{\mu_1}(\Var/\pt)\subset\Khat{\muhat}(\Var/\pt)
\]
of \textit{monodromy-free motives}, which form a sub-pre-$\lambda$-ring.  
\smallbreak
Equivalently, there is an embedding
\begin{align*}
\Khat{\GG_m,d'}(\Var/\AA_{\CC}^1)/\nu_{d'}^*\left(\Khat{}(\Var/\pt)\right)&\hookrightarrow \Khat{\GG_m,d}(\Var/\AA_{\CC}^1)/\nu_d^*\left(\Khat{}(\Var/\pt)\right)\\
[X\xrightarrow{f} \AA_{\CC}^1]&\mapsto [X\xrightarrow{(t\mapsto t^{d/d'})\circ f}\AA_{\CC}^1]
\end{align*}
and we define $\Khat{\mon}(\Var/\pt)$ to be the limit of these embeddings.  We define the isomorphic groups 
\begin{equation}
\label{KTrans}
\Xi\colon \Khat{\muhat}(\Var/Y)\cong\Khat{\mon}(\Var/Y)
\end{equation}
as a limit of quotients in the same way.  As above, there is an embedding
\begin{align*}
\Khat{}(\Var/Y)&\rightarrow \Khat{\mon}(\Var/Y)\\
[X\xrightarrow{f} Y]&\mapsto [X\xrightarrow{f\times 0} Y\times\AA_\CC^1]
\end{align*}
and we will generally abuse notation and consider elements $[X]$ and $[X\xrightarrow{f}Y]$ as elements of $\Khat{\mon}(\Var/\pt)$ and $\Khat{\mon}(\Var/Y)$, respectively, via these embeddings.

Given a morphism $h\colon Y\rightarrow Y'$ of varieties, we define the operations 
\[
h_!\colon\Khat{\mon}(\Var/Y)\rightarrow \Khat{\mon}(\Var/Y')
\]
and 
\[
h^*\colon\Khat{\mon}(\Var/Y')\rightarrow \Khat{\mon}(\Var/Y)
\]
via composition and fiber product, respectively.  We define \[\int\colon\Khat{\mon}(\Var/Y)\rightarrow\Khat{\mon}(\Var/\pt)\] by the formula 
$
\int\colonequals  (Y\rightarrow \pt)_!.
$

For $Y'\subset Y$, we define $(Y'\cap \bullet)\colonequals  (Y'\hookrightarrow Y)_!(Y'\hookrightarrow Y)^*$.  Given varieties $Y_1$ and $Y_2$, we define an external tensor product
\begin{align*}
    \boxtimes\colon &\Khat{\mon}(\Var/Y_1)\times \Khat{\mon}(\Var/ Y_2)\to \Khat{\mon}(\Var/Y_1\times Y_2)\\
    &\left([X_1\xrightarrow{f_1}Y_1\times\AAA{1}],[X_2\xrightarrow{f_2}Y_2\times\AAA{1}]\right)\mapsto [X_1\times X_2\xrightarrow{p}Y_1\times Y_2\times\AAA{1}]
\end{align*}
where $p=(\Id_{Y_1\times Y_2}\times +)\circ(f_1\times f_2)$.

The element $\LL\colonequals  [\AAA{1}]\in\Khat{}(\Var/\pt)$ has a square root
\begin{equation}\label{squareroot}
\LL^{1/2}=[\AA_{\CC}^1\xrightarrow{t\mapsto t^2}\AA_{\CC}^1]
\end{equation}
in $\Khat{\mon}(\Var/\pt)$, and we define the localized $\Khat{}(\Var/\pt)$-module
\[
\MMM{Y}\colonequals  \Khat{\mon}(\Var/Y)[\LL^{-1/2},(1-\LL^d)^{-1}\lvert\, d>0].
\]
We denote by $\MM_Y\subset \MMM{Y}$ the submodule
\begin{align}
\label{m_free_def}
\Khat{}(\Var/Y)[\LL^{-1/2},(1-\LL^d)^{-1}\lvert\, d>0].
\end{align}
We set $\MMM{}=\MMM{\pt}$.  

Note that $\sigma^2(\LL^{1/2})=0$, $\sigma^2(-\LL^{1/2})=\LL$, and more generally \[\sigma^m((-\LL^{1/2})^n)=(-\LL^{1/2})^{mn}.\]  By the results of \cite{DM11}, the operations (\ref{LambdaOps1}, \ref{LambdaOps2}) define a pre-$\lambda$-ring structure on the limit of quotients $\Khat{\mon}(\Var/\pt)$ which extends uniquely to $\MM$.  

Let $\Khat{}(\Sta/Y)$ be the group defined by symbols \eqref{effective_class}, where now $X$ is a finite type Artin stack with geometric affine stabilizers, and with relations defined as before.  One can show that $[\BGl_n]=[\Gl_n]^{-1}$ inside $\Khat{}(\Sta/\pt)$ and so there is a morphism 
\[
\MMM{Y}\rightarrow \Khat{\mon}(\Sta/Y)[\LL^{-1/2}]
\]
which is moreover an isomorphism \cite[Thm.1.2]{Ek09}.  In particular, we will be able to consider global quotient stacks $X/G$ over $Y$ as elements of $\MM_{Y}$.  Moreover, if $G$ is \textit{special} in the sense that \'etale locally trivial $G$ bundles are Zariski locally trivial, then we have
\[
[X/G\xrightarrow{f} Y]=[X\xrightarrow{f\circ p}Y]\cdot [G]^{-1}\in \MM_Y
\]
where $p\colon X\rightarrow X/G$ is the quotient map.

We define the plethystic exponential
\begin{align*}
\EXP\colon&\MMM{}\llb T_1,\ldots,T_n\rrb_+\rightarrow \MMM{}\llb T_1,\ldots,T_n\rrb\\
& \alpha\mapsto \sum_{j\geq 0}\sigma^j(\alpha)
\end{align*}
where $\MMM{}\llb T_1,\ldots,T_n\rrb_+\subset \MMM{}\llb T_1,\ldots,T_n\rrb$ is the ideal generated by $(T_1,\ldots,T_n)$.  This morphism is an isomorphism of groups onto its image $1+\MMM{}\llb T_1,\ldots,T_n\rrb_+$ (with group structure given by multiplication), with inverse denoted
\[
\LOG\colon 1+\MMM{}\llb T_1,\ldots,T_n\rrb _+\xrightarrow{\cong} \MMM{}\llb T_1,\ldots,T_n\rrb _+.
\]

More generally, let the scheme $Y=\coprod_{\gamma\in \NN^n} Y_{\gamma}$ be an infinite disjoint union of varieties.  
Then we define
\[
\MMM{Y}=\prod_{\gamma\in \NN^n}\MMM{Y_{\gamma}}.
\]
We frequently abuse notation by denoting 
\begin{align*}
\left([Z_{\gamma}\xrightarrow{f_{\gamma}}Y_{\gamma}\times\AAA{1}]\right)_{\gamma\in\NN^n}=&\left[\coprod_{\gamma\in \NN^n} Z_{\gamma}\xrightarrow{\coprod_{\gamma\in \NN^n}f_{\gamma}}Y_{\gamma}\times\AAA{1}\right]\\
=&\sum_{\gamma\in\NN^n}\left[Z_{\gamma}\xrightarrow{f_{\gamma}}Y_{\gamma}\times\AAA{1}\right].
\end{align*}

For instance, we can make the set $\NN^{n}$ into a scheme with an isolated closed point for every $\gamma\in\NN^n$; then there is a natural isomorphism
\[
\tau\colon\MMM{\NN^n}\cong\MMM{}\llb T_1,\ldots,T_n\rrb
\]
sending 
\[
\alpha\mapsto\sum_{\gamma\in \NN^n}(\{\gamma\}\hookrightarrow \NN^n)^*\alpha T^{\gamma}.
\]

Now assume that $Y$ carries a finite type monoid map $\mu\colon Y\times Y\rightarrow Y$ such that $\mu(Y_{\gamma}\times Y_{\gamma'})\subset Y_{\gamma+\gamma'}$.  We define a product on $\Khat{\mon}(\Var/Y)$:
\[
\left[X_1\xrightarrow{f_1}Y\times\AA_{\CC}^1\right]\cdot_{\mu} \left[X_2\xrightarrow{f_2}Y\times\AA_{\CC}^1\right]\colonequals \left[X_1\times X_2\xrightarrow{(\mu\times +)\circ(f_1\times f_2)}Y\times \AA_{\CC}^1\right]
\]
extending by linearity to give a product on $\MMM{Y}$.  This product is commutative if $\mu$ is.  Likewise, if $\mu$ is commutative, we define operations $\sigma_{\mu}^n$ on effective classes via
\[
\sigma_{\mu}^n\left[X\xrightarrow{f} Y\times \AA_{\CC}^1\right]\colonequals  \left[\Sym^n(X)\xrightarrow{(\mu\times +)\circ \Sym^n(f)}Y\times\AA_{\CC}^1\right]
\]
and extend to all classes as in \eqref{sigma_extend}, thus defining operations $\sigma^n$ on $\MMM{Y}$.  Set
\[
Y_+\colonequals \coprod_{\gamma\in\NN^n\setminus 0}Y_{\gamma}.
\]

We define
\begin{align*}
    \EXP_{\mu}\colon &\MMM{ Y_+}\rightarrow\MMM{Y}\\
    &\alpha\mapsto\sum_{j\geq 0}\sigma_j(\alpha).
\end{align*}
In the case in which $Y=\NN^n$, given the finite type commutative monoid structure arising from addition, this recovers the previous definition of the plethystic exponential via the natural isomorphism $\tau$.

For example, let $P$ be a variety.  Then we consider the configuration space of unordered points on $P$ \[\Sym(P)=\coprod_{i\in\NN}\Sym^i(P),\] where $\Sym^0(P)=\pt$.  There is a union map 
\[
\cup\colon\Sym^i(P)\times\Sym^j(P)\rightarrow \Sym^{i+j}(P)
\]
making $\Sym(P)$ into a commutative monoid.  By the above definitions, there is a plethystic exponential
\[
\EXP_{\cup}\colon \MMM{\Sym(P)_+}\rightarrow \MMM{\Sym(P)}
\]
along with an inverse isomorphism of groups from the image
\[
\LOG_{\cup}\colon 1+\MMM{\Sym(P)_+}\xrightarrow{\cong}\MMM{\Sym(P)_+}.
\]

\subsection{The motivic version of the main theorem}
Let $g\in \Gamma(Y)$ be a regular function on a smooth variety $Y$, and let $Y_0$ be the reduced zero locus of $g$.  In \cite{DL01} Denef and Loeser define the motivic vanishing cycle $[\phi_g]\in\Khat{\muhat}(\Var/Y_0)$; we denote by $[\phi^{\mon}_g]\in\Khat{\mon}(\Var/Y_0)$ its image under the isomorphism (\ref{KTrans}).  There is not a consensus in the literature regarding the normalizing factor for vanishing cycles;  we pick the normalization so that, if $g=0$, 
\begin{equation}
    \label{norm_fac}
[\phi^{\mon}_g]=\LL^{-\dim Y/2}[Y\xrightarrow{\id_Y}Y].
\end{equation}

Let $X$ and $Z$ be varieties, and assume that $Z\subset Y_0\cap \crit(g)$ for some fixed function $g\in\Gamma(Y)$. 
Let $f\colon Z\rightarrow X$ be a morphism of varieties, we write
\[
[Z\xrightarrow{f}X]_{\vir}\colonequals f_!(Z\hookrightarrow Y_0)^* [\phi^{\mon}_g]\in\Khat{\mon}(\Var/X).
\]

Let $Y$ be a $G$-equivariant variety, for $G$ a special algebraic group, let $Z\subset Y$ be a $G$-equivariant subvariety, let $X$ be a variety, and let $\overline{f}\colon Z\rightarrow X$ be a $G$-invariant morphism inducing a morphism of stacks \[f\colon Z/G\rightarrow X.\]
If $\overline{g}\in \Gamma(Y)^G$ is a $G$-invariant function with $Z\subset \crit(\ol{g})\cap Y_0$, we extend the above definition by setting
\[
[Z/G\xrightarrow{f} X]_{\vir}\colonequals  \overline{f}_!(Z\hookrightarrow Y_0)^*[\phi^{\mon}_{\overline{g}}]\,\LL^{\dim(G)/2}/[G].
\]

In general, motivic vanishing cycles can be hard to calculate, but things simplify in the $\GG_m$-equivariant setting thanks to a theorem of Nicaise and Payne.  
\begin{theorem}\cite[Thm.4.1.1+Prop.5.1.5]{NP17}
\label{NPThm}
Let $\overline{X}=X\times\AA_{\CC}^n$ be a $\GG_m$-equivariant variety with action 
\[
z\cdot (x,z_1,\ldots,z_n)=(x,z^{d_1}z_1,\ldots,z^{d_n}z_n)
\]
such that all the $d_i$ are non-negative.  Assume furthermore that $g\colon \overline{X}\rightarrow\AA_{\CC}^1$ is $\GG_m$-equivariant with weight $d>0$ action on the target.  Then there is an equality
\[
\pi_{X,!}[\phi^{\mon}_g]=\LL^{-\frac{\dim(\overline{X})}{2}}[\overline{X}\xrightarrow{\pi_X\times g}X\times\AAA{1}]\in\Khat{\mon}(\Var/X).
\]
Consequently, if $S\subset X$ is a subvariety and $\overline{S}=\pi_X^{-1}(S)$, there is an equality
\[
\int\left(\overline{S}\cap [\phi^{\mon}_g]\right)=\LL^{-\frac{\dim(\overline{X})}{2}}[\overline{S}\xrightarrow{g\lvert_{\overline{S}}} \AAA{1}]\in\Khat{\mon}(\Var/\pt).
\]
\end{theorem}
The motivic incarnation of our main theorem is a consequence of the Nicaise--Payne theorem.
\begin{theorem}
\label{MotDDRpre}
Let $g\in\Gamma(X\times\AA_{\CC}^n)$ satisfy the conditions of Theorem \ref{NPThm}, and assume moreover that there is a decomposition $\AA^n_{\CC}=\AA^m_\CC\times\AA^{n-m}_{\CC}$ such that we can write
\[
g=g_0+\sum_{j=1}^mg_jt_j.
\]
with $g_0,\ldots,g_m$ pulled back from $X\times\AA_\CC^{n-m}$, and $t_1,\ldots,t_m$ coordinates for $\AA_\CC^m$.  Let $\overline{Z}=Z(g_1,\ldots,g_m)$ be the vanishing locus of the functions $g_1,\ldots,g_m$, and let $g^{\red}$ be the restriction of $g_0$ to $\overline{Z}$.  Then there is an equality
\[
\pi_{X,!}[\phi^{\mon}_g]=\LL^{-\frac{\dim(\overline{X})}{2}}[\overline{Z}\xrightarrow{\pi_X\times g^{\red}} X\times\AA_\CC^1]\in \Khat{\mon}(\Var/X).
\]
In particular, if $S\subset X$ is a subvariety and $\overline{S}=\pi_X^{-1}(S)$, there is an equality
\[
\int\left(\overline{S}\cap [\phi^{\mon}_g]\right)=\LL^{-\frac{\dim(\overline{X})}{2}}[\overline{S}\cap\overline{Z}\xrightarrow{g^{\red}\lvert_{\overline{S}\cap\overline{Z}}} \AAA{1}]\in\Khat{\mon}(\Var/\pt).
\]
\end{theorem}
\begin{proof}
Set $T=X\times\AA^{n-m}_\CC$, let $\pi_T\colon \overline{X}\rightarrow T$ be the projection, and for $0\leq j\leq m$ let $h_j\in\Gamma(T)$ satisfy 
\[
h_j\pi_{T}=g_j.
\]
We stratify $T$ by setting 
\[
T_{k}=\left\{x\in T\colon\medskip h_{l}(x)=0\textrm{ for }1\leq l<k,\textrm{ and } h_{k}(x) \neq 0\right\}
\]
for $k=1,\ldots,m$ and $T_{m+1}=Z(h_1,\ldots,h_m)$.  We denote by $\overline{X}_k$ the preimage of $T_k$ under the natural projection from $\overline{X}$.  Then $\overline{Z}=\overline{X}_{m+1}$.  By Theorem (\ref{NPThm}) we have
\begin{align*}
\pi_{X,*}[\phi^{\mon}_g]=&\LL^{-\frac{\dim(\overline{X})}{2}}(\overline{X}\rightarrow X)_![\overline{X}\xrightarrow{\id_{\overline{X}}\times g}\overline{X}\times \AA_{\CC}^1]\\
=&\LL^{-\frac{\dim(\overline{X})}{2}}(T\rightarrow X)_!(\overline{X}\rightarrow T)_![\overline{X}\xrightarrow{\id_{\overline{X}}\times g}\overline{X}\times \AA_{\CC}^1]\\
=&\sum_{k=1}^{m+1}\LL^{-\frac{\dim(\overline{X})}{2}}(T_k\rightarrow X)_!(T_k\rightarrow T)^*(\overline{X}\rightarrow T)_![\overline{X}\xrightarrow{\id_{\overline{X}}\times g}\overline{X}\times \AA_{\CC}^1]\\
=&\sum_{k=1}^{m+1}\LL^{-\frac{\dim(\overline{X})}{2}}(T_k\rightarrow X)_!\left[\overline{X}_k\xrightarrow{\pi_{T_k}\times g\lvert_{\overline{X}_k}} T_k\times\AA_{\CC}^1\right]
\end{align*}
and the theorem follows from the claim that all the terms in the final sum are zero, apart from the $k={m+1}$ term.  Let $1\leq k\leq m$ and consider the commutative diagram
\[
\xymatrix{
T_k\times\AA^{m}_{\CC}\ar[rr]^-{\zeta}\ar[dr]_{\pi_{T_k}\times g}&& T_k\times\AA^{m-1}_{\CC}\times\AAA{1}\ar[dl]^{\pi_{T_k\times\AAA{1}}}
\\
&T_k\times\AA^{1}_{\CC}
}
\]
where the map
\[
\zeta(v,t_1,\ldots,t_m)=(v,t_1,\ldots,\hat{t}_k,\ldots,t_m,g(v,t_1,\ldots,t_m))
\]
has an inverse given by
\[
\zeta^{-1}(v,t_1,\ldots,\hat{t}_k,\ldots,t_{m},s)=(v,t_1,\ldots,t_{k-1},(s-g_0(v)-\sum_{j\neq k}t_jg_j(v))/g_k(v),t_{k+1},\ldots,t_m).
\]
It follows by definition that $[T_k\times\AA_{\CC}^{m-1}\times\AA_{\CC}^1\xrightarrow{\pi_{T_k\times\AA_{\CC}^1}}T_k\times\AA_{\CC}^1]$ is zero in $\Khat{\mon}(\Var/T_k)$, and so we deduce that 
\[
[\overline{X}_k=T_k\times\AA^{m}_{\CC}\xrightarrow{\pi_{T_k}\times g\lvert_{\ol{X}_k}}T_k\times\AA^{1}_{\CC}]=0
\]
for $k=1,\ldots, m$, and this implies the desired equality.
\end{proof}

\begin{theorem}
\label{MotDDR}
Following the notation of Theorem \ref{MainThm}, let $g\in\Gamma(\overline{X})^{\chi}$ be a $\GG_m$-semi-invariant function for a positive character $\chi$.  Then there is an equality

\[
\pi_{X,!}[\phi^{\mon}_g]=\LL^{-\frac{\dim(\overline{X})}{2}}[\overline{Z}\xrightarrow{\pi_X\times g^{\red}} X\times\AA_\CC^1]\in \Khat{\mon}(\Var/X).
\]
and, for $\ol{S}$ defined as before, there is an equality
\[
\int\left(\overline{S}\cap [\phi^{\mon}_g]\right)=\LL^{-\frac{\dim(\overline{X})}{2}}[\overline{S}\cap\overline{Z}\xrightarrow{g^{\red}\lvert_{\overline{S}\cap\overline{Z}}} \AAA{1}]\in\Khat{\mon}(\Var/\pt).
\]

\end{theorem}
\begin{proof}
Let $\{U_i\}_{i\in I}$ be a finite open cover of $X$ which trivializes the vector bundle $\mathcal{V}$.  For $J\subset I$ a nonempty subset, write
\[
U_J\colonequals  \begin{cases} \bigcap_{j\in J}U_j& \textrm{if }J\textrm{ is non-empty},\\
X&\textrm{otherwise}.\end{cases}
\]
Let $\pi_J$ be the restriction of $\pi$ to $\ol{U}_J\colonequals\pi^{-1}(U_J)$.  Then
\begin{align*}
\pi_{X,!}[\phim{g}]=&\sum_{J\subset I} (-1)^{\lvert J\lvert}\pi_{J,!}[\phim{g}\cap \ol{U}_J]\\
=&\sum_{J\subset I} (-1)^{\lvert J\lvert}\LL^{-\frac{\dim(\overline{X})}{2}}[\overline{Z}\cap \ol{U}_J\xrightarrow{\pi_J\times g^{\red}} X\times\AA_\CC^1]\\
=&\LL^{-\frac{\dim(\overline{X})}{2}}[\overline{Z}\xrightarrow{\pi_X\times g^{\red}} X\times\AA_\CC^1],
\end{align*}
where the second equality follows by Theorem \ref{MotDDRpre}.
\end{proof}

\section{Motivic Donaldson--Thomas invariants for quivers with potential}
In this section we prove the Cazzaniga--Morrison--Pym--Szendr\H{o}i conjecture on motivic Donaldson--Thomas invariants for $\LQ{3}$ with the deformed Weyl potential, after recalling the necessary background and definitions in \S \ref{DTdefs}.
\subsection{Motivic partition functions}
\label{DTdefs}
Let $Q$ be a quiver, by which we mean two sets $Q_1$ and $Q_0$ that we always assume to be finite, along with two maps 
\[
s,t\colon Q_1\rightarrow Q_0.
\]
The sets $Q_0$ and $Q_1$ should be thought of as the set of vertices and arrows, respectively, while the maps $s$ and $t$ take an arrow to its source and target, respectively.  

For simplicity, in this section we will only consider symmetric quivers, i.e. those quivers such that for all pairs of vertices $i,i'$ there are as many arrows from $i$ to $i'$ as from $i'$ to $i$.  

Let $\gamma\in\NN^{Q_0}$ be a dimension vector.  A representation of $Q$ with dimension vector $\gamma$ is given by a set of vector spaces $(V_i)_{i\in Q_0}$ with $\dim(V_i)=\gamma(i)$ and linear maps $\rho_a\colon V_{s(a)}\rightarrow V_{t(a)}$ for each $a\in Q_1$.  We write $\udim(V)=\gamma$.  

We define the affine space
\[
\XX_{\gamma}(Q)\colonequals  \prod_{a\in Q_1}\Hom(V_{s(a)},V_{t(a)}),
\]
which is acted on by the gauge group
\[
\Gl_{\gamma}\colonequals  \prod_{i\in Q_0} \Gl(V_i)
\]
by change of basis.  

Let $\Rep_{\gamma}(Q)$ denote the stack of $\gamma$-dimensional $Q$-representations, which we identify throughout this paper with left $\CC Q$-modules, i.e. $\Rep_{\gamma}(Q)(X)$ for a scheme $X$ is the groupoid of locally free coherent sheaves on $X$ with an action of $Q$ such that the fibers are $\gamma$-dimensional $Q$-representations.

The stack-theoretic quotient 
\[
\Mst_{\gamma}(Q)\colonequals \XX_{\gamma}(Q)/\Gl_{\gamma}
\]
is isomorphic to $\Rep_{\gamma}(Q)$. 
We denote by
\[
\Msp_{\gamma}(Q)\colonequals  \Spec(\Gamma(\XX_{\gamma}(Q))^{\Gl_{\gamma}}),
\]
the affinization of this stack, and consider the canonical map
\[
p_{\gamma}\colon \XX_{\gamma}(Q)/\Gl_{\gamma}\rightarrow \Msp_{\gamma}(Q).
\]
 At the level of geometric points, this is the map that takes a $Q$-representation over a field extension $K\supset \CC$ to its semisimplification. 
\smallbreak
Let 
\[
W\in \CC Q/[\CC Q,\CC Q]_{\mathrm{vect}}
\]
be a linear combination of cyclic paths in $Q$, i.e. a \textit{potential}.  Evaluating $\Tr(W)$ at the representations defined by points of $\XX_{\gamma}(Q)$ provides a $\Gl_{\gamma}$-invariant function $\Tr(W)_{\gamma}$ on $\XX_{\gamma}(Q)$, and so a function $\WWW_{\gamma}$ on $\Mst_{\gamma}(Q)$ pulled back from the induced function $\WW_{\gamma}$ on $\Msp_{\gamma}(Q)$.  We define
\[
\XX_{\gamma}(Q,W)\colonequals \crit\left(\Tr(W)_{\gamma}\right).
\]
Further, define the stack \[\Mst_{\gamma}(Q,W)\colonequals  \XX_{\gamma}(Q,W)/\Gl_{\gamma}.\]
Since we realize $\XX_{\gamma}(Q,W)$ and $\Mst(Q,W)$ as critical loci, we write
\begin{align*}
[\XX_{\gamma}(Q,W)]_{\text{vir}}\colonequals &\int [\phi^{\mon}_{\Tr(W)_{\gamma}}]\in \Khat{\mon}(\text{Var}/\text{pt})\\
[\Mst_{\gamma}(Q,W)]_{\text{vir}}\colonequals &\int [\phi^{\mon}_{\Tr(W)_{\gamma}}]\cdot[\Gl_{\gamma}]^{-1}\cdot \LL^{\gamma\cdot \gamma/2}\in \MMM.  
\end{align*}

For $K$ a field, we define the Jacobi algebra
\begin{align*}
K(Q,W)=K Q/\langle \partial W/\partial a \colon\medskip a\in Q_1\rangle
\end{align*}
where $\partial W/\partial a$ is the noncommutative derivative of $W$ with respect to $a$.  The critical locus of $\WWW_{\gamma}$ and the stack of $\gamma$-dimensional $\CC(Q,W)$-modules are equal as substacks of $\Mst_{\gamma}(Q)$.  

Fix a bijection $Q_0\cong\{1,\ldots,n\}$.  For $\gamma\in\mathbb{N}^{Q_0}$, we write
\[
T^{\gamma}\colonequals  T_1^{\gamma(1)}\cdot\ldots\cdot T_n^{\gamma(n)}.
\]
Then we define the \textit{motivic DT partition function}
\[
\mathcal{Z}_{Q,W}(T)\colonequals \sum_{\gamma\in\mathbb{N}^{Q_0}} [\Mst_{\gamma}(Q,W)]_{\text{vir}}\,T^{\gamma}\in1+\MMM{}\llb T_1,\ldots,T_n\rrb_+
\]
and the \textit{motivic DT invariants} $\Omega_{Q,W,\gamma}$ as in \cite{KSMot} via the formula
\begin{equation}
    \label{MDTdef}
\sum_{0\neq \gamma\in\mathbb{N}^{Q_0}}\Omega_{Q,W,\gamma}(\LL^{1/2}-\LL^{-1/2})^{-1}T^{\gamma}=\LOG(\mathcal{Z}_{Q,W}(T)).
\end{equation}

A priori, the elements $\Omega_{Q,W,\gamma}$ are only defined as elements of $\MMM{}$.  The integrality conjecture states that the elements $\Omega_{Q,W,\gamma}$ are in the image of the natural map from $\Khat{\mon}(\Var/\pt)[\LL^{-1/2}]$, see \cite{COHA, DaMe4} for proofs of variants of this conjecture. The proof of the Cazzaniga--Morrison--Pym--Szendr\H{o}i conjecture below, in particular, implies the integrality conjecture for the quivers with potential that we consider in this paper.
\subsection{Motivic DT invariants for the deformed Weyl potential}
We come now to the main application of Theorem \ref{MotDDR} in this paper.  Throughout this section we fix $Q$ to be the three loop quiver $\LQ{3}$, i.e. we fix $ Q_0=\{1\}$ and $Q_1=\{a,b,c\}$.  Let $Q'\cong\LQ{2}$ be the two loop quiver obtained by removing $a$ from $Q_1$.  Set 
\begin{equation}
\label{WDef}
W=[a,b]c.
\end{equation}
For $d\geq 2$, define
\begin{equation}
\label{WdDef}
W_d=[a,b]c+c^d.
\end{equation}
Let $\GG_m$ act on triples of matrices in $\XX_n(Q)$ via
\[
z\cdot (A,B,C)=(z^{d-1}A,B,zC).
\]
Then $\Tr(W_d)_n$ is a $\GG_m$-equivariant function on $\XX_n(Q)$, of weight $d$, and we can write 
\begin{align*}
\Tr(W_d)_n(\rho)=&\Tr(W)(\rho(a),\rho(b),\rho(c))+\Tr(\rho(c)^d)\\
=&\left(\sum_{(i,j)\in\{1,\ldots,n\}^2}\rho(a)_{ij}[\rho(b),\rho(c)]_{ji}\right)+\Tr(\rho(c)^d).
\end{align*}

In the terminology of Theorem \ref{MotDDR}, let $g_0=\Tr(\rho(c)^d)$ and let $g_1,\ldots,g_{n^2}$ be the functions recording the matrix entries of $[\rho(b),\rho(c)]$; then $Z$, the zero locus of $g_1,\ldots, g_{n^2}$, is the locus on which $\rho(b)$ and $\rho(c)$ commute, which we denote by
\begin{equation}
\label{Comm_def}
\Comm_n\colonequals \{(B,C)\in\Mat_{n\times n}(\CC)^{\times 2}\colon\medskip [B,C]=0\}\subset \Mat_{n\times n}(\CC)^{\times 2}.
\end{equation}

The group $\Gl_n$ acts on $\Comm_n$ by simultaneous conjugation, and we define
\begin{equation}
\label{StComm_def}    
\StComm_n\colonequals \Comm_n/\Gl_n
\end{equation}
to be the stack-theoretic quotient.  This is isomorphic to the stack of $n$-dimensional $\CC[y,z]$-modules.

Theorem \ref{MotDDR} gives the equality in $\Khat{\mon}(\Var/\pt)$
\begin{equation}
\label{nilpComm}    
[\XX_n(Q,W_d)]_{\vir}=\LL^{-\frac{n^2}{2}}\cdot[\Comm_n\xrightarrow{(B,C)\mapsto \Tr(C^d)} \AA_{\CC}^1].
\end{equation}

We define a map \[\lambda_n\colon\Comm_n\rightarrow\Sym^n(\AAA{2})\] sending a pair of commuting matrices $B,C$ to the generalized eigenvalues (with multiplicities) of the two matrices. More precisely, if 
\[\bb{C}^n\cong \oplus_{i\in I}V_{i}\]
is a decomposition of $\mathbb{C}^n$ into subspaces preserved by both $B$ and $C$ for which the generalized eigenvalues are $\alpha_{i,1}$ and $\alpha_{i,2}$, respectively, and such that for all $i\neq j$ either $\alpha_{i,1}\neq \alpha_{j,1}$ or $\alpha_{i,2}\neq \alpha_{j,2}$, then $(A,B)$ is mapped to the $n$-tuple of points in $\AAA{2}$ containing $(\alpha_{i,1},\alpha_{i,2})$ $\dim(V_i)$ times.  The morphism $\lambda$ is $\Gl_n$-invariant, and so induces a map \[\StComm_n\rightarrow \Sym^n(\AAA{2}),\]
which is simply the Hilbert--Chow map taking a length $n$ sheaf on $\AAA{2}$ to its support.

Denote by $Q^{\fram}$ the quiver obtained by adding a vertex $\infty$ to $Q$ and an extra arrow $j$ from $\infty$ to $1$, which we recall is the original vertex of $Q$.  We consider the potential $W$ also as a potential for $Q^{\fr}$.  Consider

\[
\ncHilb_n\subset \XX_{(1,n)}(Q^{\fram})/\Gl_n
\]
the open substack corresponding to $Q^{\fram}$-representations $\rho$ for which the image of $\rho(j)$ generates $\rho$.  This is actually a fine moduli scheme called the noncommutative Hilbert variety \cite{Nori, Rei05}.  

We define the map 
\[
\tilde{\lambda}_n\colon\XX_n(Q,W)\rightarrow \Sym^n(\AAA{3})
\]
the same way as $\lambda_n$, and denote by the same symbol $\tilde{\lambda}_n$ the extensions/restrictions to moduli stacks of $Q^{\fram}$-representations lying in the critical locus of $\Tr(W)$.

The following incarnation of the wall crossing formula is standard, but we will indicate the proof for completeness:
\begin{proposition}
\label{WCP}
There is an equality in $\MMM{\Sym(\AAA{3})}$:
\begin{align*}
&\sum_{n\geq 0}[\Mst_{(1,n)}(Q^{\fram},W)\xrightarrow{\tilde{\lambda}_n} \Sym^n(\AAA{3})]_{\vir}=\\&\sum_{n\geq 0}[\Mst_{n}(Q,W)\xrightarrow{\tilde{\lambda}_n} \Sym^n(\AAA{3})]_{\vir}\cdot\LL^{-n/2} \cdot_{\cup} \sum_{n\geq 0}[\ncHilb_{n}(Q,W)\xrightarrow{\tilde{\lambda}_n} \Sym^n(\AAA{3})]_{\vir}\cdot \frac{\LL^{1/2}}{\LL-1}.
\end{align*}
\end{proposition}
\begin{proof}
There is a decomposition into locally closed substacks
\[
\Mst_{(1,n)}(Q^{\fram})=\coprod_{i=0}^n\Mst_{(1,n)}(Q^{\fram})_{[i]}
\]
where 
\[
\Mst_{(1,n)}(Q^{\fram})_{[i]}\subset \Mst_{(1,n)}(Q^{\fram})
\]
is the substack of $Q^{\fram}$-representations $\rho$ for which the subspace generated by the image of $\rho(j)$ under the action of $\rho(a),\rho(b),\rho(c)$ is $i$-dimensional.  We define
\[
\iota_{[i]}\colon \XX_{(1,n)}(Q^{\fram})_{[i]}\hookrightarrow\XX_{(1,n)}(Q^{\fram})
\]
to be the inclusion of the subset for which the image of $\rho(j)$ lies in the first summand of the decomposition $\CC^n=\CC^i\oplus \CC^{n-i}$, and generates it, and for which $\rho(a),\rho(b),\rho(c)$ also preserve this summand.  We define $\Gl_{i,n-i}\subset \Gl(\CC^n)$ to be the subgroup of automorphisms preserving the filtration $0\subset \CC^i\subset \CC^n$.  Then
\[
\Mst_{(1,n)}(Q^{\fram})_{[i]}\cong \XX_{(1,n)}(Q^{\fram})_{[i]}/(\Gl(\CC)\times \Gl_{i,n-i})
\]
where the first factor of the gauge group acts by 
\[
z\cdot (\rho(a),\rho(b),\rho(c),\rho(j))=(\rho(a),\rho(b),\rho(c),z^{-1}\rho(j)).  
\]
We may identify 
\[
\ncHilb_n(Q)=\XX_{(1,n)}(Q^{\fram})_{[n]}/\Gl_n.
\]
The decomposition $\CC^n=\CC^i\oplus\CC^{n-i}$ induces a decomposition
\[
\XX_{(1,n)}(Q^{\fram})=V_0\times V_1\times V_{-1}
\]
with
\begin{align*}
V_0=&\XX_{(1,i)}(Q^{\fram})_{[i]}\times \XX_{n-i}(Q)\\
V_1=&\Hom(\CC^{n-i},\CC^i)^{\times 3}\\
V_{-1}=&\Hom(\CC^{i},\CC^{n-i})^{\times 3}\times \Hom(\CC,\CC^{n-i})
\end{align*}
and we have $\XX_{(1,n)}(Q^{\fram})_{[i]}=V_0\times V_1$.  Note that the diagram 
\[
\xymatrix{
\XX_{(1,n)}(Q^{\fram})_{[i]}\ar[d]^{\iota_{[i]}}\ar[r]&V_0\ar[rr]^-{\tilde{\lambda}_i\times \tilde{\lambda}_{n-i}}&&\Sym^i(\AAA{3})\times \Sym^{n-i}(\AAA{3})\ar[d]^{\cup}
\\
\XX_{(1,n)}(Q^{\fram})\ar[rrr]^{\tilde{\lambda}_n}&&&\Sym^n(\AAA{3})
}
\]
commutes.  Let $\CC^*$ act on each $V_k$ with weight $k$.  Denote by $f_{[i]}$ the restriction of $\Tr(W)_n$ to $V_0$.  Then $\Tr(W)_n$ is $\CC^*$-invariant, and by the integral identity \cite{LQT} we have
\begin{align*}
(\XX_{(1,n)}(Q^{\fram})_{[i]}\rightarrow V_0)_*\iota_{[i]}^*[\phi^{\mon}_{\Tr(W)_{(1,n)}}]=&[\phi^{\mon}_{f_{[i]}}]\cdot \LL^{(i-n)/2},
\end{align*}
where the exponent of $\LL$ is given by the difference in dimensions between $V_1$ and $V_{-1}$.

By the motivic Thom--Sebastiani theorem \cite{DL99}, we deduce that
\[
[\phi^{\mon}_{f_{[i]}}]=[\phi^{\mon}_{\Tr(W)_{(1,i)}}]\boxtimes [\phi^{\mon}_{\Tr(W)_{n-i}}]\in \Khat{\mon}(\Var/V_0).
\] 
Finally, setting 
\[
H_i=([\Gl(\CC)\times\Gl(\CC^i)\times\Gl(\CC^{n-i})])\in\Khat{}(\Var/\pt)
\]
we calculate
\begin{align*}
&[\Mst_{(1,n)}(Q^{\fram},W)\xrightarrow{\tilde{\lambda}_n} \Sym^n(\AAA{3})]_{\vir}=\\&\sum_{i=0}^n(\XX_{(1,n)}(Q^{\fram},W)_{[i]}\xrightarrow{\tilde{\lambda}_{n}}\Sym^n(\AAA{3}))_*\iota_{[i]}^*[\phi^{\mon}_{\Tr(W)_{(1,n)}}]\LL^{(1+n^2)/2}/[\Gl(\CC)\times\Gl_{i,n-i}]=\\
&\sum_{i=0}^n [\phi^{\mon}_{\Tr(W)_{(1,i)}}]\boxtimes [\phi^{\mon}_{\Tr(W)_{n-i}}]\cdot\LL^{(i-n)/2}\cdot \LL^{(1+n^2)/2}/[\Gl(\CC)\times\Gl_{i,n-i}]=\\
&\sum_{i=0}^n [\phi^{\mon}_{\Tr(W)_{(1,i)}}]\boxtimes [\phi^{\mon}_{\Tr(W)_{n-i}}]\cdot\LL^{(i-n)/2}\cdot \LL^{(1+i^2+(n-i)^2)/2}/H_i=\\
&\sum_{i=0}^n[\ncHilb_{i}(Q,W)\xrightarrow{\tilde{\lambda}_{i}} \Sym^{i}(\AAA{3})]_{\vir}\cdot_{\cup}[\Mst_{n-i}(Q,W)\xrightarrow{\tilde{\lambda}_{n-i}} \Sym^{n-i}(\AAA{3})]_{\vir}  \cdot\LL^{(i-n)/2}\cdot\frac{\LL^{1/2}}{\LL-1}
\end{align*}
as required.
\end{proof}
The next proposition is essentially a refinement of a classical result of Feit and Fine, who proved the counting result over $\mathbb{F}_q$ analogous to the identity 
\[
\sum_{n\geq 0}[\StComm_n]T^n=\EXP\left(\sum_{n\geq 1} \LL^2T^n\right)
\]
obtained by pushing forward both the left and the right hand side of \eqref{FF} to absolute motives.

\begin{proposition}
\label{2dDT}
There is an equality of generating series in $\MM_{\Sym(\AAA{2})}$:
\begin{equation}
\label{FF}    
\sum_{n\geq 0}\left[ \StComm_n\xrightarrow{\lambda_n}\Sym^n(\AAA{2})\right]=\EXP_{\cup}\left(\sum_{n\geq 1}[\AAA{2}\xrightarrow{\Delta}\Sym^n(\AAA{2})]/(\LL-1)\right).
\end{equation}

\end{proposition}
\begin{proof}

Set $W=[A,B]C$.  We define the projection
\begin{align*}
\pi_n:&\Mat_{n\times n}(\CC)^{\times 3}\rightarrow \Mat_{n\times n}(\CC)^{\times 2}\\
&(A,B,C)\mapsto (B,C).
\end{align*}
By the relative statement of the (undeformed) dimensional reduction theorem (Theorem \ref{NPThm}), there are equalities of relative motives
\begin{align*}
\pi_{n,!}[\phi^{\mon}_{\Tr(W)_n}]=&[\Mat_{n\times n}(\CC)^{\times 3}\xrightarrow{g_n}Y_n\times\AAA{1}]\LL^{-3n^2/2} &\in \Khat{\mon}(\Var/Y_n)
\\
=&[\Comm_n\hookrightarrow Y_n]\LL^{-n^2/2}&\in\Khat{}(\Var/Y_n)
\end{align*}
where $g_n(A,B,C)=(B,C,\Tr([A,B]C))$ and $Y_n=\Mat_{n\times n}(\CC)^{\times 2}$.  Accounting for the way in which we have normalized motivic vanishing cycles \eqref{norm_fac}, there is an identity
\[
[\Mst_{(1,n)}(Q^{\fram},W)\xrightarrow{\tilde{\lambda}_n}\Sym^n(\AAA{3})]_{\vir}=[\Mst_{n}(Q,W)\xrightarrow{\tilde{\lambda}_n}\Sym^n(\AAA{3})]_{\vir}\frac{\LL^{(n+1)/2}}{\LL-1}
\]
and so by Proposition \ref{WCP} we deduce that
\begin{align}
\label{relWC}
&\sum_{n\geq 0}[\ncHilb_{n}(Q,W)\xrightarrow{\tilde{\lambda}_n}\Sym^n(\AAA{3})]_{\vir}\LL^{-n/2}=\\ \nonumber
&\left(\sum_{n\geq 0}[\Mst_{n}(Q,W)\xrightarrow{\tilde{\lambda}_n} \Sym^n(\AAA{3})]_{\vir}\right)\cdot_{\cup}\left(\sum_{n\geq 0}[\Mst_{n}(Q,W)\xrightarrow{\tilde{\lambda}_n} \Sym^n(\AAA{3})]_{\vir}\LL^{-n}\right)^{-1}.
\end{align}
On the other hand, by \cite[Prop.4.3, Cor.4.4]{DR19} we deduce that
\begin{align}
    \label{DRrel}
\sum_{n\geq 0}[\ncHilb_{n}(Q,W)\xrightarrow{\tilde{\lambda}_n}\Sym^n(\AAA{3})]_{\vir}\LL^{-n/2}=\EXP_{\cup}\left(\sum_{n\geq 1}[\AAA{3}\xrightarrow{\Delta_n}\Sym^n(\AAA{3})]\frac{\LL^{-2}(1-\LL^{-n})}{1-\LL^{-1}}\right).
\end{align}
Since $\EXP\colon \MMM{\Sym(\AAA{3})_+}\rightarrow 1+\MMM{\Sym(\AAA{3})_+}$ is an isomorphism, we may write
\[
\sum_{n\geq 0}[\Mst_{n}(Q,W)\xrightarrow{\tilde{\lambda}_n} \Sym^n(\AAA{3})]_{\vir}=\EXP_{\cup}\left(\sum_{n\geq 1} \Omega_n\frac{\LL^{1/2}}{\LL-1}\right)
\]
for $\Omega_n\in\MM_{\Sym^n(\AAA{3})}$.  Then \eqref{relWC} implies that
\begin{align*}
&\sum_{n\geq 0}[\ncHilb_{n}(Q,W)\xrightarrow{\tilde{\lambda}_n}\Sym^n(\AAA{3})]_{\vir}\LL^{-n/2}=\EXP_{\cup}\left(\sum_{n\geq 1}\Omega_n\LL^{1/2}\frac{1-\LL^{-n}}{\LL-1} \right)
\end{align*}
and so $\Omega_n=\LL^{-3/2}[\AAA{3}\xrightarrow{\Delta_n}\Sym^n(\AAA{3})]$ for all $n\geq 1$ by \eqref{DRrel}.  

Let $h\colon \Sym(\AAA{3})\rightarrow \Sym(\AAA{2})$ be the map induced on tuples of points by the projection $(x,y,z)\mapsto (y,z)$.  Putting everything together,
\begin{align*}
\sum_{n\geq 0}\left[ \StComm_n\xrightarrow{\lambda_n}\Sym^n(\AAA{2})\right]=&\sum_{n\geq 0}\lambda_{n,!}\pi_{n,!}[\phi^{\mon}_{\Tr(W)_n}]\,\LL^{n^2/2}/[\Gl_n]\\
=&h_!\EXP_{\cup}\left(\sum_{n\geq 1}\LL^{-3/2}[\AAA{3}\xrightarrow{\Delta_n}\Sym^n(\AAA{3})]\frac{\LL^{1/2}}{\LL-1}\right)\\
=&\EXP_{\cup}\left(\sum_{n\geq 1}[\AAA{2}\xrightarrow{\Delta_n}\Sym^n(\AAA{2})]/(\LL-1)\right)
\end{align*}
as required.
\end{proof}

\begin{theorem}
\label{CMPSconj}
The motivic DT invariants for the quiver with potential $(Q,W_d)$ are given by the formula
\[
\Omega_{Q,W_d,n}=[\LL]^{1/2}[\AAA{1}\xrightarrow{t\mapsto t^d}\AA_\CC^1]\in\Khat{\mon}(\Var/\pt)
\]
for all $n\geq 1$.
\end{theorem}
In the case $d=3$, the above theorem is a verification of \cite[Conj.3.3]{CMPS}.  In \cite{CMPS}, the equivalent (via isomorphism \eqref{KTrans}) formulation $[\LL]^{1/2}(1-[\mu_d])$ was given for the motivic DT invariants, in the ring of $\muhat$-equivariant motives.  For $n=1$ and general $d$, the conjecture is trivial. For $n=2$ and $d=3$, its verification is already rather involved, but was successfully carried out by Le Bruyn \cite{LLB16}.  
\begin{proof}
Let $\lngth\colon\coprod_{n\geq 0}\Sym^n(\AA_{\CC}^2)\rightarrow \NN$ be the morphism of monoids taking $\Sym^n(\AAA{2})$ to the point $n$.  We define the morphism
\begin{align*}
    k_n\colon &\Sym^n(\AA_{\CC}^2)\rightarrow \AAA{1}\\
    &((x_1,y_1),\ldots(x_n,y_n))\mapsto\sum_{i=1}^ny_i^d.
\end{align*}
Then $k=\coprod_{n\geq 0}k_n$ is a morphism of commutative monoids, so that $k_!$ commutes with taking plethystic exponentials.  Combining \eqref{nilpComm} with Proposition \ref{2dDT} we deduce that 
\begin{align*}
\mathcal{Z}_{Q,W_d}(T)\colonequals  &\sum_{n\geq 0}\left([\XX_n(Q,W_d)]_{\vir}\LL^{n^2/2}/[\Gl_n]\right)T^n\\
=&\sum_{n\geq 0}\left([\Comm_n\xrightarrow{(b,c)\mapsto \Tr(c^d)} \AA_{\CC}^1]/[\Gl_n]\right)T^n&\textrm{Equation \eqref{nilpComm}}\\
=&\sum_{n\geq 0}\left(k_{n,!}\lambda_{n,!}[\Comm_n\rightarrow\Comm_n]/[\Gl_n]\right)T^n\\
=&k_!\EXP_{\cup}\left(\sum_{n\geq 1}[\AAA{2}\xrightarrow{\Delta_n}\Sym^n(\AAA{2})]/(\LL-1)\right)&\textrm{Proposition \ref{2dDT}}\\
=&\EXP\left(\sum_{n\geq 1}\left(k_{n,!}[\AAA{2}\xrightarrow{\Delta_n}\Sym^n(\AAA{2})]/(\LL-1)\right)T^n\right)\\
=&\EXP\left(\sum_{n\geq 1}\left([\AAA{2}\xrightarrow{(y,z)\mapsto nz^d}\AAA{1}]/(\LL-1)\right)T^n\right).
\end{align*}
The result then follows by comparing with \eqref{MDTdef}, since
\[
[\AAA{2}\xrightarrow{(y,z)\mapsto nz^d}\AAA{1}]/(\LL-1)=[\AAA{1}\xrightarrow{z\mapsto z^d}\AAA{1}]\LL^{1/2}/(\LL^{1/2}-\LL^{-1/2}).
\]
\end{proof}

\section{Vanishing cycles and cohomological DT theory}\label{coDTdefs}

For the rest of the paper we leave behind the naive Grothendieck ring of motives and work in the category of monodromic mixed Hodge structures and monodromic mixed Hodge modules, the natural home of \textit{cohomological} DT theory.  We introduce the key features here, for a fuller reference the reader is advised to consult \cite[Sec.2]{DaMe15b}.  

\subsection{Monodromic mixed Hodge modules}\label{mixedHodge}
Let $X$ be a complex variety.  We will work with the category of mixed Hodge modules on $X\times\AA_{\CC}^1$, as defined by Saito \cite{Sai88,Saito1,Saito89}.    For a complex variety $Y$, we denote by $\rat_Y\colon \MHM(Y)\rightarrow \Perv(Y)$ the (faithful) forgetful functor to perverse sheaves on $Y$.

We define the category of \textit{monodromic mixed Hodge modules} 
\[\MMHM(X)\colonequals \mathcal{B}_X/\mathcal{C}_X\] as the Serre quotient of two Abelian subcategories of $\MHM(X\times\AA_{\CC}^1)$.  We define $\mathcal{B}_X$ to be the full subcategory with objects those mixed Hodge modules such that for each $x\in X$ the underlying complex of constructible sheaves of 
\[\rat_{\GG_m} (x\times\GG_m\rightarrow X\times\AAA{1})^*\mathcal{F}\]
has locally constant cohomology, while $\mathcal{C}_X$ is its full subcategory with objects which satisfy the stronger condition that each $\rat_{\AAA{1}}(x\times\AAA{1}\rightarrow X\times\AAA{1})^*\mathcal{F}$ has constant cohomology sheaves.  We write $\MMHS\colonequals \MMHM(\pt)$.  

We denote by $\Dub(\MMHM(X))$ the derived category of (not necessarily bounded) complexes of monodromic mixed Hodge modules.

Objects in $\MMHM(X)$ have a weight filtration inherited from Saito's weight filtration of objects in $\MHM(X\times\AA_{\CC}^1)$, and we say that an object $\mathcal{F}\in\Ob(\Dub(\MMHM(X)))$ is \textit{pure} if $\Ho^j(\mathcal{F})\in\Ob(\MMHM(X))$ is pure of weight $j$ for all $j\in\mathbb{Z}$.  
\smallbreak

Fix a finite quiver $Q$.  As in \S \ref{DTdefs} we denote by 
\[
\Msp_{\gamma}(Q)\colonequals  \Spec\left(\Gamma(\XX_{\gamma}(Q))^{\Gl_{\gamma}}\right)
\]
the coarse moduli space of $\gamma$-dimensional $Q$-representations.  We denote by 
\[\Msp^{\stab}_{\gamma}(Q)\subset \Msp_{\gamma}(Q)\]
the smooth irreducible open subvariety of simple modules, which is dense if it is nonempty.  The closed points of $\Msp_{\gamma}(Q)$ are in bijection with semisimple $\gamma$-dimensional $\mathbb{C}Q$-modules.  We set
\[
\Msp(Q)\colonequals  \coprod_{\gamma\in\mathbb{N}^{Q_0}}\Msp_{\gamma}(Q),
\]
a monoid in the category of schemes with monoid structure denoted $\oplus$, as at the level of closed points it takes a pair of semisimple $\mathbb{C}Q$-modules to their direct sum.  The map $\oplus$ is finite by \cite[Lem.2.1]{MeRe14}.  

We denote by 
\[
\Dblf(\MMHM(\Msp(Q)))\subset\Dub(\MMHM(\Msp(Q)))
\]
the full subcategory containing those objects such that for each $\gamma$ and each weight $n$, $\Gr_W^n(\mathcal{F}\lvert_{\Msp_{\gamma}})$ has bounded total cohomology, and for each dimension vector $\gamma\in\mathbb{N}^{Q_0}$,
there is an equality $\Gr_W^n(\mathcal{F}\lvert_{\Msp_{\gamma}})=0$ for $n\gg 0$.  There is a symmetric monoidal product defined on $\Dblf(\MMHM(\Msp(Q)))$ by
\[
\mathcal{F}\boxtimes_{\oplus}\mathcal{G}\colonequals  \left(\Msp\times\AAA{1}\times \Msp\times\AAA{1}\xrightarrow{(\rho,t,\rho',t')\mapsto(\rho\oplus\rho', t+t')}\Msp\times\AAA{1}\right)_*(\mathcal{F}\boxtimes\mathcal{G})
\]
which is exact and preserves weights by \cite[Prop.3.5]{DaMe15b}.  

Let $z\colon\{0\}\rightarrow \AA_{\CC}^1$ be the inclusion of the origin.  Then the functor $(\id_{\Msp(Q)}\times z)_*$ provides an embedding of symmetric monoidal categories \[\MHM(\Msp(Q))\hookrightarrow\MMHM(\Msp(Q))\] which moreover preserves weights.  In this way we consider $\MHM(\Msp(Q))$ as a full symmetric monoidal subcategory of $\MMHM(\Msp(Q))$.
Let \begin{equation}\label{defLLL}
    \LLL\colonequals  z_*\HO_c(\AA_{\CC}^1,\QQ).
\end{equation}
Then $\LLL$ is concentrated in cohomological degree 2, and its second cohomology is a pure weight 2 one-dimensional monodromic mixed Hodge structure.  Moreover, there is a tensor square root to $\LLL$, defined by
\begin{equation}\label{defLLLh}
\LLL^{1/2}:=(\AAA{1}\xrightarrow{x\mapsto x^2}\AA_{\CC}^1)_*\QQ_{\AA_{\CC}^1}.\end{equation}
Note that there is no tensor square root for $\LLL$ inside $\Dblf(\MHS)$, since a pure odd-weight Hodge structure must have even dimension.

\subsection{Vanishing cycles}
\label{vancycles}
Let $Y$ be a smooth variety, and consider a regular function
\[f:Y\to\AAA{1}.\]
Let $\kappa:Y_0\to Y$ be the inclusion of the zero fiber of $f$, and consider the pullback diagram induced by $\text{exp}:\AAA{1}\to\AAA{1}$
\[
\begin{tikzcd}
\widetilde{Y} \arrow{r}{} \arrow{d}{p} & \AAA{1} \arrow{d}{\text{exp}}\\%
Y \arrow{r}{f} & \AAA{1}. 
\end{tikzcd}
\]
\\

Define the nearby cycle functor $\psi_f: \Dub(\Perv(Y))\to \Dub(\Perv(Y))$ by the formula
\[\psi_f\colonequals \kappa_*\kappa^*p_*p^*. \]

The vanishing cycle functor $\varphi_f[-1]\colon
\Dub(\Perv(Y))\to \Dub(\Perv(Y))$ sends perverse sheaves to perverse sheaves.
For $\mc{F}\in \Dub(\Perv(Y))$, $\varphi_f\mc{F}$ fits in the distinguished triangle
\[\kappa_*\kappa^*\mc{F}\to \psi_f \mc{F}\to \varphi_f \mc{F}. \]

In \cite{Sai88,Saito89}, Saito defines an upgrade of the nearby and vanishing cycle functors to functors
\[
\uppsi_f[-1], \phi_f[-1]\colon\MHM(Y)\rightarrow \MHM(Y)
\]
for the category of mixed Hodge modules.  These are upgrades in the sense that there are natural isomorphisms
\begin{align*}
&\rat_Y\circ \uppsi_f[-1]\cong \psi_f[-1]\circ \rat_Y\\
&\rat_Y\circ \phi_f[-1]\cong \varphi_f[-1]\circ \rat_Y.
\end{align*}
We denote by the same symbol the functor $\phi_f\colon\Dub(\MHM(Y))\rightarrow \Dub(\MHM(Y))$.  We define
\begin{equation}
\label{phimdef}
\phim{f}\colonequals (Y\times\GG_m\rightarrow Y\times\mathbb{A}^1)_!\phi_{f/u}(Y\times\GG_m\xrightarrow{\pi_Y} Y)^*\colon\MHM(Y)\rightarrow\MMHM(Y).
\end{equation}
Since there is a natural isomorphism 
\begin{equation}    
\label{commTate}
\phi^{\mon}_f(\mathcal{F}\otimes\LLL^{n/2})\cong\phi^{\mon}_f\mathcal{F}\otimes\LLL^{n/2} 
\end{equation}
for $n$ even, we can use the right hand side of \eqref{commTate} to define the left hand side when $n$ is odd.  

There is a forgetful functor $\forg:\MMHM(X)\rightarrow \Perv(X)$ defined as follows.  Denote by $t$ the coordinate for $\AAA{1}$.  We denote by the same symbol the induced function on $X\times\AAA{1}$.  Then the vanishing cycle functor restricts to give a functor
\[
\phi_t[-1]:\mathcal{B}_X\rightarrow \MHM(X)
\]
where we have identified $X$ with the zero locus of $t$.  Since all objects in $\mathcal{C}_X$ are sent to the zero object by this functor, there is a unique functor \[\overline{\phi_t}[-1]:\MMHM(X)\rightarrow \MHM(X)\] 
through which $\phi_t[-1]$ factors.  Composing with the faithful forgetful functor $\rat_X:\MHM(X)\rightarrow \Perv(X)$, we obtain the faithful functor
\[
\forg_X=\rat_X\circ\overline{\phi_t}[-1]:\MMHM(X)\rightarrow \Perv(X),
\]
see \cite[Sec.2.1]{DaMe15b} for further details.
\\

We next explain how to extend the definition of vanishing cycles to smooth quotient stacks. Let $Y$ be a smooth variety with an action of a reductive group $G\subset \Gl_n$, and consider a $G$-invariant function $f\colon Y\to\AAA{1}$. For $N\geq n$, define $\Fr(n,N)\subset (\AAA{N})^n$ the open set of $n$-tuples of linearly independent vectors in $\AAA{N}$. The action of $G$ on $\Fr(n, N)$ is scheme-theoretically free, so $Y\times_G \text{Fr}(n, N)$ is a smooth variety.  Let $f_N\in\Gamma(Y\times_G \text{Fr}(n, N))$ be the function induced by $f$.  
Then there are isomorphisms, which we can take as a definition for the purposes of this paper:
\begin{equation}\label{equiv}
\HO^j_c\left(Y/G, \phi^{\text{mon}}_f\mathbb{Q}\right)\cong \lim_{N\rightarrow \infty}\HO^j_c\left(Y\times_G \text{Fr}(n, N), \phi^{\text{mon}}_{f_N}\mathbb{Q}\otimes \LLL^{-nN}\right).    
\end{equation}
Let $p\colon Y\rightarrow X$ be a $G$-invariant morphism.  Let $p_N\colon Y\times_G\Fr(n,N)\rightarrow X$ be the induced morphism.  Then there are isomorphisms, which we again take to be the definition:
\begin{equation}\label{rel_equiv}
\Ho^j(X, p_!\phi^{\text{mon}}_f\mathbb{Q})\cong \lim_{N\rightarrow \infty}\Ho^j(p_{N,!} \phi^{\text{mon}}_{f_N}\mathbb{Q}\otimes \LLL^{-nN}).    
\end{equation}

\subsection{Cohomological BPS invariants.}\label{coBPS_defs}
Here we briefly review the definition of cohomological BPS invariants for a quiver with potential, as well as the cohomological integrality theorem that expresses the cohomology of the vanishing cycle cohomology of the stack of $\CC(Q,W)$-modules in terms of the cohomological BPS invariants.

Let $Q$ be a quiver and let $W\in\CC Q/[\CC Q,\CC Q]$ be a potential.  For $N\in\NN$, we define the quiver $Q_N$ to be the quiver obtained from $Q$ by adding one extra vertex, labelled $\infty$, and $N$ arrows from $\infty$ to $i$ for each $i\in Q_0$.  We identify 
\[
\NN^{(Q_N)_0}=\NN \oplus \NN^{Q_0}.
\]
Fix a dimension vector $\gamma\in \NN^{Q_0}$.  Let 
\[
\XX^{\stab}_{(1,\gamma)}(Q_N)\subset\XX_{(1,\gamma)}(Q_N)
\]
be the open subvariety with closed points corresponding
to $\CC Q_N$-modules $\rho$ such that there are no proper submodules $\rho'\subset \rho$ with $\dim(\rho')_{\infty}=1$.
The $\Gl_{\gamma}$-action on this variety is scheme-theoretically free, and we define
\[
\Msp^{\fram}_{N,\gamma}(Q)\colonequals \XX^{\stab}_{(1,\gamma)}(Q_N)/\Gl_{\gamma},
\]
the fine moduli space of $\gamma$-dimensional stable $N$-framed modules.  This space is smooth, and the forgetful map 
\[
q_{N,\gamma}\colon\Msp^{\fram}_{N,\gamma}(Q)\rightarrow\Msp_{\gamma}(Q)
\]
is proper. We continue to write $\WW_{\gamma}$ for the function defined by $\Tr(W)$ on $\Msp_{\gamma}(Q)$, and we write $\WW_{N,\gamma}$ for the function defined on $\Msp^{\fram}_{N,\gamma}(Q)$.  

There are isomorphisms
\begin{align*}
\Ho^j(p_{\gamma,!}\phim{\WWW_{\gamma}}\QQ_{\Mst_\gamma(Q)})\cong&\lim_{N\rightarrow \infty} \Ho^j\left(q_{N,\gamma,!}\phim{\WW_{N,\gamma}}\QQ_{\Msp_{N,\gamma}^{\fram}(Q)}\otimes \LLL^{-N \cdot\sum_{i\in Q_0} \gamma_i}\right)
\\
\HO_c^j(\Mst_{\gamma}(Q),\phim{\Tr(W)_{\gamma}}\QQ_{\Mst_\gamma(Q)})\cong&\lim_{N\rightarrow \infty}\HO_c^j\left( \Msp_{N,\gamma}^{\fram}(Q),\phim{\WW_{N,\gamma}}\QQ_{\Msp_{N,\gamma}^{\fram}(Q)}\otimes \LLL^{-N \cdot\sum_{i\in Q_0} \gamma_i}\right).
\end{align*}
As a very special case, letting $Q$ be the quiver with one vertex and no loops, and taking the dimension vector $(1)$, we calculate
\begin{align*}
\HO_c(\pt/\CC^{*})=&\lim_{N\rightarrow \infty}\left(\HO(\CC\mathbb{P}^{N-1},\QQ)\otimes\LLL^{-N}\right)\\
=&\bigoplus_{j\leq -1}\LLL^j.
\end{align*}

It follows from the properness of the maps $q_{N,\gamma}$ \cite[Prop.4.4]{DaMe15b} that there is an isomorphism
\begin{align}
\label{decomp_thm}
\HO_c(\Mst_{\gamma}(Q),\phim{\WWW_{\gamma}}\QQ_{\Mst_\gamma(Q)})\cong &\HO_c\left(\Msp_{\gamma},\Ho\left(p_!\phim{\WWW_\gamma}\QQ_{\Mst_\gamma}\right)\right).
\end{align}

By the cohomological integrality theorem \cite[Thm.A]{DaMe15b}, there is an isomorphism of monodromic mixed Hodge modules
\begin{equation}
\label{CohInt}
\bigoplus_{\gamma\in\mathbb{N}^{Q_0}}\Ho\left(p_{\gamma,!}\phim{\Tr(W)}\QQ_{\Mst_\gamma}\right)\otimes\LLL^{(\gamma,\gamma)/2}\cong \Sym_{\boxtimes_{\oplus}}\left(\BPS_{Q,W,\gamma}\otimes\HO_c(\pt/\CC^{*})_{\vir}\right).
\end{equation}
Here we define
\[
\HO_c(\pt/\CC^{*})_{\vir}=\HO_c(\pt/\CC^{*})\otimes\LLL^{1/2}
\]
and
\[
\BPS_{Q,W,\gamma}\colonequals  \begin{cases} \phim{\WW_\gamma}\ICS_{\Msp_\gamma(Q)}(\QQ)\otimes \LLL^{-\dim(\Msp_{\gamma}(Q))/2}& \textrm{if }\Msp^{\stab}_{\gamma}(Q)\neq \emptyset\\
0&\textrm{otherwise.}\end{cases}
\]
Here, $\ICS_{\Msp_\gamma(Q)}(\QQ)$ is (up to shifting cohomological degree down by $\dim(\Msp_{\gamma}(Q))$) the intersection complex mixed Hodge module on $\Msp_{\gamma}(Q)$ obtained by taking the intermediate extension of the constant mixed Hodge module on $\Msp^{\stab}_{\gamma}(Q)$,  i.e. 
\[
\rat_{\Msp_{\gamma}(Q)}\ICS_{\Msp_\gamma(Q)}(\QQ)[\dim(\Msp_{\gamma}(Q))]
\]
is the simple perverse sheaf on $\Msp_{\gamma}(Q)$ extending $\QQ_{\Msp^{\stab}_{\gamma}(Q)}[\dim(\Msp^{\stab}_{\gamma}(Q))]$. 

It then follows from \eqref{decomp_thm} that
\begin{equation}
\label{coh_int_thm}    
\bigoplus_{\gamma\in\mathbb{N}^{Q_0}}\HO_c(\Mst_{\gamma}(Q),\phim{\Tr(W)}\QQ_{\Mst_\gamma(Q)})\otimes\LLL^{\chi_Q(\gamma,\gamma)/2}\cong \Sym\left(\BPSA{Q,W,\gamma}\otimes\HO_c(\pt/\CC^{*})_{\vir}\right),
\end{equation}
where for dimension vectors $d,e\in \mathbb{N}^{Q_0}$, the pairing $\chi_Q(d,e)$ is defined by
\[\chi_Q(d,e):=\sum_{i\in Q_0}d_ie_i-\sum_{a\in Q_1}d_{s(a)}e_{t(a)}\]
and the $\BPSA{Q,W,\gamma}$ invariants, duals of the BPS invariant considered in \cite[Thm.A]{DaMe15b}, are defined by the formula
\[
\BPSA{Q,W,\gamma}\colonequals \HO_c(\Msp_{\gamma}(Q),\BPS_{Q,W,\gamma}).
\]
\smallbreak

Finally, for the connection to the the motivic sections of the paper, there is a ring homomorphism
\begin{align*}
\chi_{\MMHS}\colon\Khat{\muhat}(\Var/\pt)&\rightarrow \Grot(\MMHS)\\
[X]&\mapsto -[(X\times_{\mu_d}\GG_m\xrightarrow{(x,t)\mapsto t^d}\mathbb{A}_{\mathbb{C}}^1)_!\QQ_{X\times_{\mu_d}\GG_m}]
\end{align*}
taking the motivic DT invariants to the Hodge theoretic DT invariants, and we have 
\[
\chi_{\MMHS}(\Omega_{Q,W,\gamma})=[\BPSA{Q,W,\gamma}]_{\KK_0},
\]
see \cite[Sec.2.7]{Dav19a} for details.
\section{Cohomological DT invariants for the deformed Weyl potential}
\label{DefWeylCoh}
This section is devoted to proving Theorem \ref{cohdef}.  Throughout the section we fix $Q=\LQ{3}$, the three loop quiver with loops labelled $a,b,c$, and with potential as in \eqref{WdDef}: \[W_d=[a,b]c+c^d\]
for $d\geq 2$. 

Our task is to determine the BPS sheaves \[\BPS_{Q,W_d,n}\in \MMHM(\Msp_n(Q))\]
along with the (dual) BPS cohomology $\BPSA{Q,W_d,n}$ as defined in \S \ref{coBPS_defs}.  We follow the strategy of \cite{Dav16b}; we prove that the monodromic mixed Hodge modules $\BPS_{Q,W_d,n}$ are pure, have very restricted support, and are moreover constant on their support.  Due to these facts, it is enough to calculate the class of $\BPSA{Q,W_d,n}$ in the Grothendieck group of monodromic mixed Hodge structures.  The result then follows from our earlier motivic calculations, specifically Theorem \ref{CMPSconj}.
\begin{lemma}
\label{prefac}
Let $\rho$ be a representation of $\CC(Q,W_d)$.  Then each of the operators $\rho(a),\rho(b),\rho(c)$ preserve the generalized eigenspaces of each of the others.  Moreover, the only nontrivial generalized eigenspace for $\rho(c)$ is for the generalized eigenvalue zero.
\end{lemma}
\begin{proof}
The operator $\rho(c)$ is nilpotent, following the proof of \cite[Lem.3.7+Lem.3.9]{CMPS}, and so $\rho(c)$ has only one generalized eigenspace, which is trivially preserved by $\rho(a)$ and $\rho(b)$.  The Jacobi relations include the relations
\begin{align*}
\partial W_d/\partial a=[b,c]\\
\partial W_d/\partial b=[c,a]
\end{align*} 
and so it follows that $\rho(c)$ preserves the generalized eigenspaces of $\rho(a)$ and $\rho(b)$ as well.  Let $v$ be a generalized eigenvector of $\rho(b)$, with generalized eigenvalue $\lambda$.  Define $\beta=(\rho(b)-\lambda\cdot)$.  Since \[[\rho(a),\beta]=d \rho(c)^{d-1},\]
it follows that $[\rho(a),\beta]$ commutes with $\beta$, and so for $m\gg 0$
\begin{align*}
\beta^m\rho(a)v=&\rho(a)\beta^mv+ m[\beta,\rho(a)]\beta^{m-1}v\\
=&0.
\end{align*}
This means that $\rho(a)v$ is a generalized eigenvector for the operator $\rho(b)$ with generalized eigenvalue $\lambda$.  The same argument, swapping $\rho(a)$ and $\rho(b)$, shows that $\rho(b)$ preserves the generalized eigenspaces of $\rho(a)$.
\end{proof}
\begin{corollary}
\label{CanDec}
Every finite-dimensional $\CC(Q,W_d)$-module $\rho$ admits a canonical decomposition into nonzero $\CC(Q,W_d)$-modules
\[
\rho\cong\bigoplus_{s\in \Sigma}\rho_s,
\]
where $\Sigma\subset \CC^2$ is a finite subset and, for $s=(s_1,s_2)$, the generalized eigenvalues of $\rho(a),\rho(b)$ and $\rho(c)$ restricted to $\rho_s$ are given by $s_1$, $s_2$, and $0$, respectively.
\end{corollary}

\begin{definition}
\label{GEDef}
For $\rho$ a $\CC(Q,W_d)$-module, we call the set $\Sigma$ in Corollary \ref{CanDec} the set of generalized $(a,b)$-eigenvalues of $\rho$.
\end{definition}

\begin{lemma}\label{half_support}
Let $\Msp^{\cnilp}_n\subset \Msp_n(Q)$ be the closed subvariety corresponding to those $\CC Q$-modules for which $c$ acts via the zero map.  Then $\supp(\BPS_{Q,W_d,n})\subset \Msp^{\cnilp}_n$.
\end{lemma}
\begin{proof}
By \eqref{CohInt}, there is an inclusion 
\[
\BPS_{Q,W_d,n}\otimes\LLL^{1/2}\hookrightarrow \Ho\left(p_{n,!}\phim{\TTTr(W_d)_n}\QQ_{\Mst_n(Q)}\otimes\LLL^{(\gamma,\gamma)/2}\right),
\]
and so $\supp(\BPS_{Q,W_d,n})\subset p_n(\crit{\TTTr(W_d)_n})$.  In particular, for a $\CC Q$-module $\rho$ corresponding to a point in $\supp(\BPS_{Q,W_d,n})$, $\rho(c)$ acts nilpotently, and commutes with the action of $\rho(a)$ and $\rho(b)$.  On the other hand, such modules are semisimple (as they correspond to points of $\Msp_n(Q))$, and so it follows that $\rho(c)$ acts via the zero map.
\end{proof}
In fact, we can significantly strengthen Lemma \ref{half_support}.  Consider the inclusion 
\begin{align}\label{Deltadef}
\Delta_n\colon &\AA_{\CC}^2\hookrightarrow \Msp_n(Q)\\\nonumber
&(x,y)\mapsto (x\cdot \Id_{n\times n},y\cdot \Id_{n\times n},0).
\end{align}
\begin{lemma}
\label{supplemm}
There is an inclusion $\supp(\BPS_{Q,W_d,n})\subset \Delta_n(\AA_{\CC}^2)$. Furthermore, there exists $\mathcal{G}_n\in\Ob(\MMHS)$ such that
\[
\BPS_{Q,W_d,n}\cong \Delta_{n,*}\bb{Q}_{\AA^2}\otimes\mathcal{G}_n\otimes\LLL^{-1}.
\]
\end{lemma}
The proof of this lemma is essentially the same as the proof of \cite[Lem.4.1]{Dav16b}; we give an abridged version of the proof.
\begin{proof}
By Lemma \ref{prefac}, any finite-dimensional representation $\rho$ of $\CC(Q,W_d)$ splits canonically as a direct sum  of nonzero representations
\begin{equation}
\label{rhodec}
\rho=\bigoplus_{s\in \Sigma}\rho_{s}
\end{equation}
where $\Sigma\subset \CC^2$ is a finite subset and the generalized eigenvalue of the operators $\rho(a)\lvert_{\rho_{(\lambda_1,\lambda_2)}}$ and $\rho(b)\lvert_{\rho_{(\lambda_1,\lambda_2)}}$ are $\lambda_1$ and $\lambda_2$, respectively.  If we assume moreover that $\rho$ is semisimple, then $\rho(c)=0$ since $\rho(a)$ and $\rho(b)$ preserve $\ker(\rho(c))$.  It follows that $\rho(a)$ and $\rho(b)$ commute, and so since $\rho$ is semisimple, $\rho(a)$ and $\rho(b)$ are simultaneously diagonalizable and
\[
\supp\left(\Ho\left(p_!\phim{\TTTr(W_d)}\QQ_{\Mst(Q)}\right)\right)\subset \Sym\left(\coprod_{n\geq 1}\Delta_n(\AAA{2})\right).
\]
For an analytic open subset $U\subset\AAA{2}$, let $\Msp^{U}\subset \Msp(Q,W)$  be the open analytic subspace of semisimple $\CC(Q,W_d)$-modules $\rho$ such that in the (minimal) decomposition (\ref{rhodec}), we have that $\Sigma\subset U$.  Let $U_1$ and $U_2$ be disjoint open analytic subsets of $\AA_{\CC}^2$. Then \[p^{-1}\Msp^{U_1\coprod U_2}=p^{-1}\Msp^{U_1}\times p^{-1}\Msp^{U_2},\] and via the Thom--Sebastiani isomorphism there is a natural isomorphism 
\[
\Ho\left(p_!\phim{\TTTr(W_d)}\QQ_{\Mst(Q)}\right)\lvert_{\Msp^{U_1\coprod U_2}}\cong \Ho\left(p_!\phim{\TTTr(W_d)}\QQ_{\Mst(Q)}\right)\lvert_{\Msp^{U_1}}\boxtimes_{\oplus}\Ho\left(p_!\phim{\TTTr(W_d)}\QQ_{\Mst(Q)}\right)\lvert_{\Msp^{U_2}}.
\]
Using the cohomological integrality theorem, we obtain an isomorphism
\begin{align*}
&\Sym_{\boxtimes_{\oplus}}\left(\bigoplus_{n\geq 1}\BPS_{Q,W_d,n}\lvert_{\Msp^{U_1\coprod U_2}}\otimes\HO_c(\pt/\CC^*)_{\vir}\right)\cong\\&\Sym_{\boxtimes_{\oplus}}\left(\bigoplus_{n\geq 1}\BPS_{Q,W_d,n}\lvert_{\Msp^{U_1}}\otimes\HO_c(\pt/\CC^*)_{\vir}\right)\boxtimes_{\oplus}\Sym_{\boxtimes_{\oplus}}\left(\bigoplus_{n\geq 1}\BPS_{Q,W_d,n}\lvert_{\Msp^{U_2}}\otimes\HO_c(\pt/\CC^*)_{\vir}\right)\cong\\
&\Sym_{\boxtimes_{\oplus}}\left(\bigoplus_{n\geq 1}\big(\BPS_{Q,W_d,n}\lvert_{\Msp^{U_1}}\oplus \BPS_{Q,W_d,n}\lvert_{\Msp^{U_1}}\big)\otimes\HO_c(\pt/\CC^*)_{\vir}\right).
\end{align*}
This implies that $\BPS_{Q,W_d,n}\lvert_{\Msp^{U_1\coprod U_2}}\cong\BPS_{Q,W_d,n}\lvert_{\Msp^{U_1}}\oplus\BPS_{Q,W_d,n}\lvert_{\Msp^{U_2}}$.  

Unravelling this a little: if $\rho$ is a semisimple module lying in the support of $\BPS_{Q,W_d,n}$, for which all of the generalized $(a,b)$-eigenvalues lie in $U_1\coprod U_2$, then either all of the generalized $(a,b)$-eigenvalues lie in $U_1$ or they all lie in $U_2$. 
It follows that it is not possible to separate the generalized $(a,b)$-eigenvalues of any $\rho$ lying in the support of $\BPS_{Q,W_d,n}$ into two open sets in $\AAA{2}$, and so in fact they must all be the same, i.e. the decomposition (\ref{rhodec}) can have only one summand, which is the part of the lemma regarding support.
\smallbreak
For the second part of the lemma, let $\Mat^0_{n\times n}(\CC)\subset \Mat_{n\times n}(\CC)$ denote the subspace of trace-free matrices, and set
\[
\XX_n^0\colonequals  \Mat^0_{n\times n}(\CC)^{\times 2}\times\Mat_{n\times n}(\CC)\subset\XX_n(Q).
\]
There is a $\Gl_n$-equivariant isomorphism
\begin{align*}
\AAA{2}\times \XX_n^0&\rightarrow \XX_{n}(Q)\\
(x,y,A,B,C)&\mapsto (x\cdot\Id_{n\times n}+A,y\cdot \Id_{n\times n}+B,C).
\end{align*}
The $\Gl_n$-action on $\AA_{\CC}^2$ is trivial, and is the conjugation action on all of the other factors.  It follows that 
\[
\Msp_n(Q)\cong\AAA{2}\times \Msp^0_n
\]
where
\[
\Msp^0_n\colonequals  \Spec\left(\Gamma(\XX^0_n)^{\Gl_n}\right).
\]
Further, we have that 
\begin{align*}
\ICS_{\Msp_n(Q)}(\QQ)\cong&\QQ_{\AAA{2}}\boxtimes\ICS_{\Msp^0_n}(\QQ)\\
\phim{\TTr(W_d)_n}\ICS_{\Msp_n(Q)}(\QQ)\cong &\QQ_{\AAA{2}}\boxtimes\left(\phim{\TTr(W_d)_n}\ICS_{\Msp^0_n}(\QQ)\right).
\end{align*}
The second isomorphism follows from the fact that the function $\TTr(W_d)_n\in\Gamma(\Msp_n(Q))$ factors through the projection to $\Msp^0_n$. The condition on the support of $\BPS_{Q,W_d}$ implies that $\phim{\TTr(W_d)_n}\ICS_{\Msp^0_n}(\QQ)$ is supported at $0\in\Msp^0_n$, and the second part of the lemma follows.
\end{proof}
\begin{lemma}
\label{purlemm}
For all $n$, the monodromic mixed Hodge module $\BPS_{Q,W_d,n}$ is pure.
\end{lemma}
\begin{proof}
From the proof of the previous lemma, it is enough to show that the monodromic mixed Hodge module $\phim{\TTr(W_d)_n}\ICS_{\Msp^0_n}(\QQ)$ is pure. This complex of monodromic mixed Hodge modules is supported at a point, so it is enough to show that $\HO\left(\Msp^0_n,\phim{\TTr(W_d)_n}\ICS_{\Msp^0_n}(\QQ)\right)$ is pure. 
For this we use the main geometric result of \cite{DMSS13}: $\ICS_{\Msp^0_n}(\QQ)$ is a pure complex of mixed Hodge modules, $\TTr(W_d)_n\colon\Msp^0_n\rightarrow \CC$ is a $\GG_m$-equivariant function, and the support of $\phim{\TTr(W_d)_n}\ICS_{\Msp^0_n}(\QQ)$ is proper (since it is a point), so the cohomology $\HO\left(\Msp^0_n,\phim{\TTr(W_d)_n}\ICS_{\Msp^0_n}(\QQ)\right)$ is a pure complex of monodromic mixed Hodge structures by \cite[Thm.3.1]{DMSS13}.
\end{proof}
\begin{proofof}{Theorem \ref{cohdef}}
The existence of the isomorphism \eqref{CMPSA} follows from the existence of the isomorphism \eqref{CMPSR}, and due to Lemma (\ref{supplemm}) is in fact equivalent to it.  By Lemma \ref{purlemm}, the isomorphism class of $\BPSA{Q,W_d,n}$ is determined by its class in the Grothendieck group of monodromic mixed Hodge structures, which by Theorem \ref{CMPSconj} is equal to 
\[
\chi^{\MMHS}(\LL^{1/2}[\AAA{1}\xrightarrow{t\mapsto t^d}\AA_\CC^1])=[\HO_c(\AAA{1},\phim{t^d}\QQ_{\AAA{1}})\otimes\LLL^{1/2}]_{\KK_0},
\]
as required.
\end{proofof}

\section{Cohomological deformed dimensional reduction}
\label{MainThmSec}
In this section we prove Theorems \ref{MainThmCor} and \ref{MainThm}.  In fact, Theorem \ref{MainThmCor} is the special case of Theorem \ref{MainThm} in which the short exact sequence theorem \eqref{sescoh} is a split exact sequence of direct sums of $\OO_X$.  Conversely, we have the following

\begin{proposition}\label{2implies3}
 If, under the conditions of Theorem \ref{MainThmCor}, \eqref{prime3} is always an isomorphism, then under the conditions of Theorem \ref{MainThm}, \eqref{DRI} is an isomorphism.
\end{proposition}

\begin{proof}
The question of whether the natural map \eqref{DRI} is an isomorphism is local on $X$, and we can cover $X$ with open affine subvarieties $E$ such that there is an isomorphism of short exact sequences
\[
\xymatrix{
0\ar[r]& \mathscr{V}'_E\ar[r]\ar[d]^{\cong}&\mathscr{V}_E\ar[r]\ar[d]^{\cong}&\mathscr{V}''_E\ar[r]\ar[d]^{\cong}&0\\
0\ar[r]&\OO_E^{\oplus (n-m)}\ar[r]&\OO_E^{\oplus n}\ar[r]& \OO_E^{\oplus m}\ar[r]&0.
}
\]
The lower short exact sequence is split, and so locally we are in the setup of Theorem \ref{MainThmCor}.
\end{proof}

Accordingly, we will spend most of this section proving that \eqref{prime3} is an isomorphism.  So we assume that we have a function $g\in \Gamma(\overline{X})$, a decomposition $\overline{X}=X\times\AAA{n}$, and a decomposition $\AAA{n}=\AAA{m}\times\AAA{n-m}$ satisfying the $\GG_m$-equivariance assumptions of Theorem \ref{MainThmCor}.  As in the statement of Theorem \ref{MainThmCor}, we denote by $\pi\colon \overline{X}\rightarrow X$ the projection.

We assume that $m=1$, since the general case follows from this. Define 
\[X'\colonequals X\times\AAA{n-1}.\]

For $z\in\CC$, define
\[
\ol{X}_z=(X'\times\AAA{1})_z\colonequals  g^{-1}(z).
\]

By assumption, we can decompose 
\[g=g_0+tg_1,\] where $t$ is the coordinate on $\AAA{m}=\AAA{1}$, and $g_0$ and $g_1$ are functions pulled back from $X'$. 
Let $Z\subset X'$ be the zero locus of $g_1$.  We define
\[h\colonequals g^{\red}\colonequals g_0\lvert_{Z\times\AAA{1}}\]
and
\begin{align*}
(Z\times\AAA{1})_z\colonequals  &h^{-1}(z)\subset Z\times\AAA{1}.
\end{align*}

Further, define the inclusions for $z\in\mathbb{C}$
\begin{align*}
i\colon &Z\times\AAA{1}\to X'\times\AAA{1}=\overline{X}\\
j\colon &\ol{X}\setminus\ol{X}_0\rightarrow \ol{X}\\
i_z\colon &(Z\times\AAA{1})_z\to (X'\times\AAA{1})_z\\
\iota_z\colon &(Z\times\AAA{1})_z\to Z\times \AAA{1}\\
\kappa_z\colon &(X'\times\AAA{1})_z\to X'\times\AAA{1}.
\end{align*}

These fit into the commutative diagram
\[
\begin{tikzcd}
(Z\times\AAA{1})_0 \arrow{d}[swap]{\iota_0} \arrow{r}{i_0} & (X'\times\AAA{1})_0 \arrow{d}[swap]{\kappa_0}\\
Z\times\AAA{1} \arrow{r}{i} & X'\times\AAA{1} \\
(Z\times\AAA{1})_1 \arrow{u}{\iota_1} \arrow{r}{i_1} & (X'\times\AAA{1})_1. \arrow{u}{\kappa_1}
\end{tikzcd}
\]

Let $\mathcal{F}\in\Ob(\MHM(X))$.  The natural map $\pi^*\mathcal{F}\to i_{*}i^*\pi^*\mathcal{F}$ induces a map 
\begin{equation}
\pi_!\phim{g}\pi^*\mathcal{F}\to \pi_!i_{*}\phim{h}i^*\pi^*\mathcal{F}    
\end{equation}
in $\Db(\MHM(X))$, which we wish to show is an isomorphism.  By faithfulness of the forgetful functor $\forg_X$, it is sufficient to show that the morphism
\begin{equation}
    \label{target_ni}
\pi_!\varphi_{g}\pi^*\mathcal{F}\to \pi_!i_{*}\varphi_{h}i^*\pi^*\mathcal{F},    
\end{equation}
considered as a morphism in the derived category of constructible sheaves, is an isomorphism.

Taking duals, this is equivalent to showing that the following map of complexes of constructible sheaves is an isomorphism
\begin{equation}
\label{finalv}
\pi_*i_{*}\varphi_{h}i^!\pi^*\mathcal{E}\to \pi_*\varphi_{g}\pi^*\mathcal{E},
\end{equation}
where $\mathcal{E}=\bb{D}\mathcal{F}[2n]$.  So we will spend the rest of this section showing that \eqref{finalv} is an isomorphism in the derived category, for $\mathcal{E}$ a bounded complex of constructible sheaves on $X$.

We introduce some more notation that will be used in this section. Consider the diagram of Cartesian squares  
\[
\begin{tikzcd}
\widetilde{Z\times\AAA{1}} \arrow{r}{\widetilde{i}} \arrow{d}{s} & \widetilde{X'\times\AAA{1}} \arrow{d}{p} \arrow{r} & \AAA{1} \arrow{d}{\text{exp}}\\
Z\times\AAA{1} \arrow{r}{i} & X'\times\AAA{1} \arrow{r}{g}& \AAA{1}
\end{tikzcd}
\]
in which the right square is the diagram used to define the vanishing cycle functor for the function $g$, see \S \ref{vancycles}.



Before we start the proof of Theorem \ref{MainThmCor}, we establish some preliminary results.
We say that a sheaf $\mathcal{F}$ on a space with a $\GG_m$-action is \textit{locally constant on $\GG_m$-orbits} if 
the restriction $\mathcal{F}|_O$ is locally constant
for any $\GG_m$-orbit $O$.


\begin{proposition}\label{vanishing_prop}
Let $S$ be a complex variety.  Let $T=S\times\AAA{n}$ be the $\GG_m$-equivariant variety obtained by taking the product of $S$, with the trivial action, and $\AAA{n}$ acted on by some non-negative weights.
Denote by $\pi\colon T\to S$ the projection. 
Let $g\colon  T\to \AAA{1}$
be a $\GG_m$-equivariant map,
where $\GG_m$ acts with nonzero weight on $\AAA{1}$. 
Denote by  $j\colon  T\setminus T_0\to T$ the open immersion of the complement of $T_0\colonequals g^{-1}(0)$. 

Let $\mathcal{F}$ be a sheaf on $T\setminus T_0$ locally constant on $\GG_m$-orbits. 
Then \[\pi_*j_!\mathcal{F}=0.\]

\end{proposition}
\begin{proof}
First decompose $\AAA{n}=\AAA{n'}\times \AAA{n''}$ where $\GG_m$ acts with strictly positive weights on $\AAA{n'}$ and acts trivially on $\AAA{n''}$.  Then we can decompose $\pi=\pi''\pi'$ where $\pi'\colon T\rightarrow S\times \AAA{n''}$ and $\pi''\colon  S\times\AAA{n''}\rightarrow S$ are the natural projections.  We deduce that $\pi_*j_!\cong \pi''_*\pi'_*j_!$ and so it is enough to prove that $\pi'_*j_!\mathcal{F}=0$, i.e. it is enough to prove the proposition under the stronger assumption that all of the $\GG_m$-weights on $\AAA{n}$ are strictly positive.

Consider the (stacky) weighted projective space
\[
V\colonequals (T\setminus S\times\{0\})/\GG_m 
\]
along with the weighted blowup
\[
B\colonequals ((T\setminus S\times\{0\})\times\AAA{1})/\GG_m
\]
where $\GG_m$ acts via the given action on $T$ and the weight $-1$ action on $\AAA{1}$.  This is a Deligne--Mumford stack, and there is an open embedding
\[
T\setminus (S\times\{0\})\cong ((T\setminus (S\times\{0\}))\times\GG_m)/\GG_m\hookrightarrow B. 
\]
We denote by $q\colon B\rightarrow V$ the rank one affine fibration and $r\colon B\rightarrow T$ the proper morphism, which is defined as follows: Let $A$ be a variety, let $P\rightarrow A$ be the total space of a principal $\GG_m$-bundle, and let $P\rightarrow (T\setminus (S\times\{0\}))\times\AAA{1}$ be a $\GG_m$-equivariant map.  Let $E\times \GG_m\rightarrow E$ be a local trivialization of the principal bundle, so that we have a morphism $E\rightarrow (T\setminus (S\times\{0\}))\times\AAA{1}$ coming from the embedding of $E$ as $E\times\{1\}$.  We postcompose this morphism with the action morphism $(T\setminus (S\times\{0\}))\times\AAA{1}\rightarrow T$, to get a morphism $E\rightarrow T$.  These morphisms then glue to give a morphism $A\rightarrow T$.

Consider the commutative diagram
\[
\xymatrix{
&T\setminus T_0\ar[dl]_{\widetilde{j}}\ar[d]^j\\
B\ar[d]^q\ar[r]^-r& T\ar[d]^{\pi}\\
V\ar[r]^-{\pi'}&S
}
\]
where $\widetilde{j}$ is the unique morphism through which $j$ factors, which exists since $S\subset T_0$.  
There are natural isomorphisms
\[
\pi_*j_!\simeq \pi_*r_*\widetilde{j}_!\simeq \pi'_*q_*\widetilde{j}_!
\]
and so it is sufficient to prove that $q_*\tilde{j}_!\mathcal{F}=0$.  
Since the statement is local on $V$, we can replace $V
$ with an open subvariety $E'$ for which $q^{-1}(E')\cong E'\times\AAA{1}$ and the restriction of $q$ to $q^{-1}(E')$ is the projection.  Now the statement is a special case of \cite[Lemma A.3]{DAV}. 
\end{proof}

The above proposition allows us to compare the cohomology of sheaves on $X\times\AAA{n}$ or $Z\times\AAA{n}$ that are locally constant on $\GG_m$-orbits with their restrictions onto the zero fiber.  The following corollary makes precise the applications that we make of this fact.

\begin{corollary}\label{cor1}
Using the notation introduced before Proposition \ref{vanishing_prop}, the vertical arrows in the diagrams
\[
\begin{tikzcd}
\pi_*\pi^*\mathcal{E} \arrow{r} \arrow{d}{\cong} & \pi_*p_*p^*\pi^*\mathcal{E} \arrow{d}{\cong}\\
\pi_*\kappa_{0*}\kappa_0^*\pi^*\mathcal{E} \arrow{r} & \pi_*\kappa_{0*}\kappa_0^*p_*p^*\pi^*\mathcal{E}
\end{tikzcd}
\]
and 
\[
\begin{tikzcd}
\pi_*i_{*}i^!\pi^*\mathcal{E}\arrow{d}{\cong} \arrow{r} & \pi_*i_{*}s_*s^*i^!\pi^*\mathcal{E} \arrow{d}{\cong}\\
\pi_*i_{*}\iota_{0*}\iota_0^*i^!\pi^*\mathcal{E} \arrow{r}& \pi_*i_{*}\iota_{0*}\iota_0^*s_*s^*i^!\pi^*\mathcal{E}
\end{tikzcd}
\]
are isomorphisms.
\end{corollary}
\begin{proof}
The complexes of sheaves $\pi^*\mathcal{E}$, $p_*p^*\pi^*\mathcal{E}$, $i^!\pi^*\mathcal{E}$ and $s_*s^*i^!\pi^*\mathcal{E}$ are constant on $\GG_m$-orbits after restricting to $\ol{X}\setminus \ol{X}_0$ and $(Z\times \AAA{1})\setminus (Z\times\AAA{1})_0$, respectively.  Since $j$ is the inclusion of the complement to $\ol{X}_0$ in $\ol{X}$, by Proposition \ref{vanishing_prop}, there is an isomorphism \[\pi_*j_!j^{*}\pi^*\mathcal{E}\cong 0.\] In the distinguished triangle
\[
\pi_*j_!j^{ *}\pi^*\mathcal{E}\rightarrow \pi_*\pi^*\mathcal{E} \rightarrow \pi_*\kappa_{0!}\kappa_0^{*}\pi^*\mathcal{E}
\]
the second morphism is thus an isomorphism.  The other three claims follow similarly.
\end{proof}


The next proposition helps compare the nearby cycle cohomology and the restriction to a nonzero fiber for a sheaf locally constant on $\GG_m$-orbits.

\begin{proposition}
\label{nby_prop}
Let $T$ be a $\GG_m$-equivariant variety, and let $\pi\colon T\to S$ be a morphism of varieties that is constant on $\GG_m$-orbits.
Consider a $\GG_m$-equivariant function $g\colon T\to\AAA{1}$, where the $\GG_m$-action on $\AAA{1}$ has nonzero weight $d$.
Denote by $\kappa_1\colon  T_1\to T$ the inclusion of the fiber of $g$ over $1$, and by $p$ the pullback of the map $\exp\colon\AAA{1}\rightarrow\AAA{1}$ along $g$.

Let $\mathcal{E}$ be a sheaf on $T\setminus T_0$ locally constant on $\GG_m$-orbits.
There exists a natural map $p_*p^*\mathcal{E}\to \kappa_{1*}\kappa_1^*\mathcal{E}$ which induces an isomorphism
\[\pi_*p_*p^*\mathcal{E}\xrightarrow{\cong} \pi_*\kappa_{1*}\kappa_1^*\mathcal{E}.\]
\end{proposition}
\begin{proof}
Consider the diagram 
\[
\xymatrix{
\wt{T'}\ar[dr]^k\ar@/^1.0pc/[drr]^{h}\ar@/_1.0pc/[ddr]_{l}\\
&\wt{T_1}\ar[d]^{p_1}\ar[r]^{\wt{\kappa_1}}&\wt{T}\ar[d]^p\ar[r]^a&\AAA{1}\ar[d]^{\exp}\\
& T_1\ar[r]^{\kappa_1}& T\ar[r]^g & \AAA{1}
\\
}
\]
Here $\widetilde{T}$ and $\widetilde{T_1}$ are defined in such a way that the two squares are Cartesian.
Further, $\widetilde{T}'$ is defined to be the fiber over zero of the map $a\colon \widetilde{T}\to \AAA{1}$, so the map $l\colon \widetilde{T}'\to T_1$ is an isomorphism. Since the squares are Cartesian, the morphism $k$ is uniquely determined.

We define $\alpha$ to be the composition of morphisms 
\[
p_*p^*\mathcal{E}\rightarrow p_*h_*h^*p^*\mathcal{E}\cong \kappa_{1*}l_*l^*\kappa_1^*\mathcal{E}\cong \kappa_{1*}\kappa_1^*\mathcal{E}.
\]


The map 
\begin{align*}
m\colon \widetilde{T}'\times\AAA{1}\to &\widetilde{T}\\
(y,z)\mapsto&(e^{z/d}y,z)\in T\times_{\AAA{1}}\AAA{1}
\end{align*}
is an isomorphism. In the following diagram, for which the sub-diagram of uncurved arrows is commutative, we use $m$ to identify $\widetilde{T}'\times\AAA{1}$ and $\widetilde{T}$.
Then $h$ is the inclusion of the zero fiber of the trivial $\AAA{1}$-bundle, while we define $\varpi$ to be the projection onto the $\tilde{T}'$ factor
\[
\xymatrix{
\wt{T'}\ar@/_.5pc/[r]_-{h}\ar[d]^t &\ar[l]_-{\varpi}\wt{T'}\times\AAA{1}\ar[d]^p\\
S&\ar[l]_{\pi} T.
}
\]
In the above diagram, the morphism $t$ is defined by $t=\pi \kappa_1 l$. 
Since $\varpi$ is a projection with contractible fibers, the natural transformation
\[
\varpi_*(\id\rightarrow h_*h^*)\varpi^*
\]
is an isomorphism. We need to show that the following natural map is an isomorphism, as it is isomorphic to $\pi_*\alpha$:
\[  t_*\varpi_*p^*\mathcal{E}\to t_*\varpi_*h_*h^*p^*\mathcal{E}.  \]
Since we assume that $\mathcal{E}$ is locally constant on $\GG_m$-orbits, there is a sheaf $\mathcal{G}$ on $\wt{T'}$ such that $p^*\mathcal{E}\cong\varpi^*\mathcal{G}$.
In the commutative diagram
\[
\xymatrix{
t_*\varpi_*p^*\mathcal{E}\ar[r]\ar[d]^{\cong}& t_*\varpi_*h _*h^*p^*\mathcal{E}\ar[d]^{\cong}\\
t_*\varpi_*\varpi^*\mathcal{G}\ar[r]^-{\cong}& t_*\varpi_*h_*h^*\varpi^*\mathcal{G},
}
\]
the top horizontal morphism is an isomorphism since the other three are.
\end{proof}

\begin{corollary}\label{cor2}
Using the notation introduced before Proposition \ref{vanishing_prop}, the following diagram commutes, where the horizontal arrows are isomorphisms:
\[
\begin{tikzcd}
\pi_*i_{*}\iota_{1*}\iota_1^*i^!\pi^*\mathcal{E} \arrow{d} & \pi_*i_{*}s_*s^*i^!\pi^*\mathcal{E}\arrow{d} \arrow{l}[swap]{\cong} \arrow{r}{\cong} &  \pi_*i_{*}\psi_{h}i^!\pi^*\mathcal{E}\arrow{d}\\
\pi_*\kappa_{1*}\kappa_1^*\pi^*\mathcal{E} & \pi_*p_*p^*\pi^*\mathcal{E}\arrow{r}{\cong} \arrow{l}[swap]{\cong}&  \pi_*\psi_g\pi^*\mathcal{E}.
\end{tikzcd}
\]
\end{corollary}

\begin{proof}
The left square clearly commutes. Its horizontal arrows are isomorphisms by Proposition \ref{nby_prop} because the sheaves $\pi^*\mathcal{E}$ on $X\times\AAA{n}$ and $i^!\pi^*\mathcal{E}$ on $Z\times\AAA{n}$ are constant on $\GG_m$-orbits. 

For the right square, observe that by the definition of nearby cycles we have that
\begin{align*}
\pi_*i_{*}\psi_{h}i^!\pi^*\mathcal{E}=
\pi_*i_{*}\iota_{0*}\iota_0^*s_*s^*i^!\pi^*\mathcal{E}\\
\pi_*\psi_g\pi^*\mathcal{E}=\pi_*\kappa_{0*}\kappa_0^*p_*p^*\pi^*\mathcal{E}.
\end{align*}

We can thus rewrite the right square as follows
\[
\begin{tikzcd}
\pi_*i_{*}s_*s^*i^!\pi^*\mathcal{E}\arrow{d} \arrow{r} &
\pi_*i_{*}\iota_{0*}\iota_0^*s_*s^*i^!\pi^*\mathcal{E}\arrow{d}\\
\pi_*p_*p^*\pi^*\mathcal{E}\arrow{r}&
\pi_*\kappa_{0*}\kappa_0^*p_*p^*\pi^*\mathcal{E}.
\end{tikzcd}
\]
\\
The square clearly commutes and its horizontal maps are isomorphisms by Corollary \ref{cor1}.

\end{proof}

\begin{lemma}
\label{shift_lemma}
Let $T$ be a variety with a $\GG_m$-action, and let $g\colon T\rightarrow\AAA{1}$ be a homogeneous regular function.
Denote by $\kappa_1\colon T_1\rightarrow T$ the inclusion of the fiber over $1$. Consider a sheaf $\mathcal{E}$ locally constant on $\GG_m$-orbits. Then there is a natural isomorphism $$\kappa_1^*\mathcal{E}\cong\kappa_1^!\mathcal{E}[2].$$ 
\end{lemma}

\begin{proof}
We use the commutative diagram from the proof of Proposition \ref{nby_prop}. Since $l$ is an isomorphism, the lemma follows from the claim that there is a natural isomorphism
\[
(\kappa_1 l)^*\mathcal{E}\cong (\kappa_1 l)^!\mathcal{E}[2].
\]
Since $p$ is locally a homeomorphism, we have $p^{!}\mathcal{E}\cong p^*\mathcal{E}$. Recall the maps
\[ 
\xymatrix{
\wt{T'}\ar@/_.5pc/[r]_-{h} &\ar[l]_-{\varpi}\wt{T'}\times\AAA{1}
}
\]
from the proof of Proposition \ref{nby_prop}. The sheaf $\mathcal{E}$ is locally constant on $\GG_m$-orbits, so there exists a sheaf $\mathcal{G}$ on $\wt{T'}$ such that $\mathcal{E}=\varpi^*\mathcal{G}$. The desired isomorphism follows now from 
$h^*\varpi^*\cong h^!\varpi^*[2]$.
\end{proof}


\begin{proofof}{Theorems \ref{MainThmCor} and \ref{MainThm}}
We first show that \eqref{prime3} and \eqref{DRI} are isomorphisms.  By Proposition \ref{2implies3}, it suffices to prove that \eqref{prime3} is.
For a constructible complex of sheaves $\mathcal{G}$ on $\overline{X}$, we have distinguished triangles
\begin{align*}
    i_*i^!\mathcal{G}\rightarrow\mathcal{G}\rightarrow j_*j^*\mathcal{G}\\
    \kappa_{0*}\kappa_0^*\mathcal{G}\rightarrow \psi_g\mathcal{G}\rightarrow\varphi_g\mathcal{G}.
\end{align*}
Furthermore, there are natural equivalences
\begin{align*}
    \psi_gi_*\cong i_*\psi_h\\
    \varphi_gi_*\cong i_*\varphi_h
\end{align*}
and so we obtain a commutative diagram, in which the rows and the columns are all distinguished triangles
\begin{equation}
    \label{3square}
\begin{tikzcd}
\pi_*i_*i^!\pi^*\mathcal{E}\arrow{r}\arrow{d}&\pi_* i_*\psi_hi^!\pi^*\mathcal{E}\arrow{r}\arrow{d}&\pi_* i_*\varphi_h i^!\pi^*\mathcal{E}\arrow{d}{\eqref{finalv}}\\
\pi_*\kappa_{0*}\kappa^*_0\pi^*\mathcal{E}\arrow{r}\arrow{d}&\pi_*\psi_g\pi^*\mathcal{E}\arrow{r}\arrow{d}&\pi_*\varphi_g\pi^*\mathcal{E}\arrow{d}\\
\pi_*\kappa_{0*}\kappa_0^*j_*j^*\pi^*\mathcal{E}\arrow{r}{(A)}&\pi_*\psi_gj_*j^*\pi^*\mathcal{E}\arrow{r}&\pi_*\varphi_gj_*j^*\pi^*\mathcal{E}.
\end{tikzcd}
\end{equation}
Since our goal is to show that \eqref{finalv} is an isomorphism, it is sufficient to show that $(A)$ is.  Via Corollary \ref{cor2}, the top left square in \eqref{3square} is isomorphic to the top square in
\begin{equation}
\begin{tikzcd}
\pi_*i_*i^!\pi^*\mathcal{E}\arrow{r}\arrow{d}&\pi_*i_*\iota_{1*}\iota_1^*i^!\pi^*\mathcal{E}\arrow{d}\\
\pi_*\pi^*\mathcal{E}\arrow{r}\arrow{d}&\pi_*\kappa_{1*}\kappa_1^*\pi^*\mathcal{E}\arrow{d}\\
\pi_*j_*j^*\pi^*\mathcal{E}\arrow{r}{(B)}&\pi_*\kappa_{1*}\kappa_1^*j_*j^*\pi^*\mathcal{E}
\end{tikzcd}
\end{equation}
and so it is sufficient to show that $(B)$ is an isomorphism.  Consider the diagram
\[
\begin{tikzcd}
(Z\times\AAA{1})_1 \arrow{d}{\iota_1} \arrow{r}{i_1} & (X'\times\AAA{1})_1 \arrow{d}{\kappa_1}& (U\times\AAA{1})_1 \arrow{l}[swap]{j_1} \arrow{d}{u_1}\\
Z\times\AAA{1} \arrow{r}{i} & X'\times\AAA{1} & U\times\AAA{1}. \arrow{l}[swap]{j}
\end{tikzcd}
\]
Via Lemma \ref{shift_lemma}, for $\mathcal{G}$ a complex of constructible sheaves on $X'\times\AAA{1}$ which is locally constant on $\GG_m$-orbits, the base change morphism 
\[\alpha\colon \kappa^*_{1}j_*\mathcal{G}\rightarrow j_{1*}u_1^*\mathcal{G}\]
is an isomorphism since it fits into the commutative square of isomorphisms
\[
\begin{tikzcd}
\kappa^*_{1}j_*\mathcal{G}\arrow{r}{\alpha}\arrow{d}&j_{1*}u_1^*\mathcal{G}\arrow{d}\\
j_{1*}u_1^!\mathcal{G}[-2]\arrow{r}&\kappa^!_{1}j_*\mathcal{G}[-2]
.
\end{tikzcd}
\]
So we have reduced the problem to proving that the morphism
\[
\pi_*j_*j^*\pi^*\mathcal{E}\rightarrow \pi_*j_*u_{1*}u_1^*j^*\pi^*\mathcal{E}
\]
is an isomorphism.  This holds because $\pi^*\mathcal{E}$ is constant along the fibers of $\pi$ and $u_1$ is a homotopy equivalence on each fiber of $\pi$. 

This completes the proof that the morphism \eqref{finalv} is an isomorphism, and so all that is left is to prove that \eqref{prime2} is an isomorphism.  For this, let 
\begin{align*}
    r\colon S\hookrightarrow X\\
    \ol{r} \colon\ol{S}\hookrightarrow \ol{X}
\end{align*} be the inclusions, and let $\tau\colon X\rightarrow \pt$ be the structure morphism.  Since the structure morphism for $\ol{X}$ can be written as $\tau \pi$, the morphism \eqref{prime2} can be written as the top horizontal arrow in the commutative diagram
\[
\xymatrix{
\tau_!\pi_!\ol{r}_!\ol{r}^*\phim{g}\pi^*\QQ_X\ar[r]\ar[d]^{\cong}&  \tau_!\pi_!\ol{r}_!\ol{r}^*i_*\phim{g^{\red}}i^*\pi^*\QQ_X\ar[d]^{\cong}
\\
\tau_!r_!r^*\pi_!\phim{g}\pi^*\QQ_X\ar[r]^-{\cong}&  \tau_!r_!r^*\pi_!i_*\phim{g^{\red}}i^*\pi^*\QQ_X
}
\]
in which the vertical arrows are isomorphisms by base change, and the bottom horizontal arrow is an isomorphism by the first part of the theorem.
\end{proofof}

As in the case of undeformed cohomological dimensional reduction, we can easily generalize Theorem \ref{MainThm} to stacks, i.e. the following corollary is a generalization of \cite[Cor.A.9]{DAV}.
\begin{corollary}
\label{DDRstack}
Let $G$ be an algebraic group, and let $X$ be a $G$-equivariant variety with $\overline{X}$ the total space of a $G$-equivariant bundle over $X$ with projection map $\pi\colon\overline{X}\rightarrow X$.  Let $\overline{g}\in \Gamma(\overline{X})^G$, and let $g\in\Gamma(\overline{X}/G)$ be the induced function on the stack.
Assume in addition the $\GG_m$-equivariance assumptions of Theorem \ref{MainThm}, and define $Z$ and $\overline{Z}$ as in that theorem.
Let $S\subset X$ be a $G$-invariant subvariety, then there is a natural isomorphism of cohomologically graded monodromic mixed Hodge structures
\[
\HO_c(\overline{S}/G,\phim{g}\QQ_{\overline{X}/G})\cong \HO_c\left((\overline{Z}\cap \overline{S})/G,\phim{g^{\red}}\QQ_{\overline{Z}/G}\right).
\]
\end{corollary}
\begin{proof}
We use the notation in \S \ref{coDTdefs} and formula (\ref{equiv}), and try to reduce clutter by fixing
\begin{align*}
    \QQ_{\ol{A}}=&\QQ_{\overline{X}\times_G \text{Fr}(n,N)}\\
    \QQ_{\ol{B}}=&\QQ_{\overline{Z}\times_G \text{Fr}(n,N)}\\
    \QQ_{A}=&\QQ_{X\times_G \text{Fr}(n,N)}.
\end{align*}

We denote by $f\in \Gamma(\ol{X}\times_G\Fr(n,N))$ the function induced by $g$, and by $f^{\red}$ the restriction to $\ol{Z}\times_G\Fr(n,N)$.  The corollary follows from the claim that the natural map
\begin{equation}
\label{alpha_map}    
\HO_c^j(\overline{S}\times_G \text{Fr}(n,N),\phim{f}\QQ_{\ol{A}}\otimes\LLL^{-nN})\rightarrow \HO_c^j\left((\overline{Z}\cap \overline{S})\times_G\text{Fr}(n,N),\phim{f^{\red}}\QQ_{\ol{B}}\otimes\LLL^{-nN}\right)
\end{equation}
is an isomorphism.  Consider the commutative diagram, where the morphisms $r$ and $\overline{r}$ are the natural inclusions
\[
\begin{tikzcd}
\overline{S}\times_G \Fr(n,N)\arrow{r}{\overline{r}}\arrow{d}& \overline{X}\times_G \Fr(n,N)\arrow{d}{\pi}&\arrow{l}[swap]{\overline{i}}\ol{Z}\times_G\Fr(n,N)\\
S\times_G\Fr(n,N)\arrow{r}{r}& X\times_G\Fr(n,N)\arrow{d}{\tau}\\
&\pt.
\end{tikzcd}
\]

Then \eqref{alpha_map} is obtained by applying $\tau_!\pi_!$ to the morphism $\overline{r}_*\overline{r}^*\phim{g_N}(\mathbb{Q}_{\ol{A}}\rightarrow \overline{i}_*\overline{i}^*\mathbb{Q}_{\ol{A}})$, and so by proper base change and the isomorphism $\pi^*\QQ_{A}\cong\QQ_{\ol{A}}$, it is sufficient to prove that applying $\tau_!r_*r^*\pi_!$ to the morphism
\[
\beta\colon \phim{g_N}(\pi^*\QQ_{A}\rightarrow \overline{i}_*\overline{i}^*\pi^*\QQ_{A})
\]
gives an isomorphism.  By Theorem \ref{MainThm}, $\pi_!\beta$ is an isomorphism, and we are done.
\end{proof}

\section{Applications}
\subsection{Vanishing cycles on preprojective stacks}
\label{ppBPS_sec}
Let $Q$ be a finite quiver.  We define $\overline{Q}$ to be the doubled quiver, i.e. $\overline{Q}$ has the same vertex set as $Q$, and we set $\ol{Q}_1=Q_1\coprod Q_1^{\mathrm{op}}$ where $Q_1^{\mathrm{op}}$ contains an arrow $a^*$ for each arrow $a\in Q_1$, with the reverse orientation.  We define $\wt{Q}$ to be the quiver with the same vertices as $Q$, and with 
\[
\wt{Q}_1\colonequals\ol{Q}_1\coprod \{\omega_i\colon\medskip i\in Q_0\}
\]
where $s(\omega_i)=t(\omega_i)=i$. 
Consider the preprojective algebra
\[
\Pi_Q\colonequals \CC\ol{Q}/\langle \sum_{a\in Q_1}[a,a^*]\rangle.
\]
For $\gamma\in\NN^{Q_0}$ we denote by $\Mst_{\gamma}(\Pi_Q)\subset \Mst_{\gamma}(\ol{Q})$ the substack of $\ol{Q}$-representations that are representations of the preprojective algebra.

We have a commutative square
\[
\xymatrix{
\Mst_{\gamma}(\wt{Q})\ar[r]^-{\pi}\ar[d]^{q}&\Mst_{\gamma}(\ol{Q})\ar[d]^-{p}\\
\Msp_{\gamma}(\wt{Q})\ar[r]^-{\varpi}&\Msp_{\gamma}(\ol{Q})
}
\]
where $q$ and $p$ are the affinization maps and $\pi$ and $\varpi$ are the forgetful maps.  The morphism $\pi$ is the projection from the total space of a vector bundle. 
The map $\varpi$ has
an $\AAA{1}$-family of sections 
\[l\colon\Msp(\overline{Q})\times\AAA{1}\rightarrow\Msp(\wt{Q})\]
given by setting the action of all of the $\omega_i$ to be multiplication by $z\in\AAA{1}$.  
We define
\[
\wt{W}\colonequals \sum_{i\in Q_0}\omega_i\sum_{a\in Q_1}[a,a^*].
\]
By \cite[Lem.4.1]{Dav16b} there are monodromic mixed Hodge modules \[\BPS_{\Pi_Q,\gamma}\in\Ob(\MMHM(\Msp_{\gamma}(\Pi_Q)))\] such that \[\BPS_{\wt{Q},\wt{W},\gamma}\cong l_*(\BPS_{\Pi_Q,\gamma}\boxtimes\QQ_{\AAA{1}})\otimes\LLL^{-1/2}.\]
We define 
\begin{equation}
    \label{undef1}
\BPSA{\Pi_Q,\gamma}\colonequals\HO_c\left(\Msp_{\gamma}(\Pi_Q),\BPS_{\Pi_Q,\gamma}\right).
\end{equation}
It follows that
\begin{align}\nonumber
    \bigoplus_{\gamma\in\NN^{Q_0}}\HO_c(\Mst_{\gamma}(\Pi_Q),\QQ_{\Mst_{\gamma}(\Pi_Q)})\otimes \LLL^{\chi_Q(\gamma,\gamma)}\cong &\bigoplus_{\gamma\in\NN^{Q_0}}\HO_c(\Mst_{\gamma}(\wt{Q}),\phim{\mathfrak{T}r(\wt{W})}\QQ_{\Mst_{\gamma}(\wt{Q})})\otimes \LLL^{\chi_{\wt{Q}}(\gamma,\gamma)/2}\\
    \cong &\Sym\left( \bigoplus_{\gamma\in\NN^{Q_0}\setminus \{0\}} \BPSA{\Pi_Q,\gamma}\otimes\HO_c(\pt/\CC^*)\otimes\LLL\right).\label{undef2}
\end{align}
where the first isomorphism is via dimensional reduction, and the second is the integrality isomorphism \cite[Thm.A]{DaMe15b}.




Let $W'\in \CC\ol{Q}/[\CC\ol{Q},\CC\ol{Q}]$ be a potential.  We consider $W'$ also as a potential for $\wt{Q}$ via the natural embedding of quivers.  We say that $\wt{W}+W'$ is \textit{quasihomogeneous} if there is a weight function $\wt{Q}_1\rightarrow \NN$ such that the weight of each cyclic word in $\wt{W}+W'$ is a strictly positive constant.

\begin{proofof}{Theorem \ref{DGT_compare}}
Fix a dimension vector $\gamma\in\NN^{Q_0}$. Define $\PV_{\gamma}\subset \XX_{\gamma}(\ol{Q})$ to be the subspace of tuples $(\rho(b))_{b\in \ol{Q}_1}$ of matrices satisfying the matrix-valued equation 
\[\sum_{a\in Q_1}[\rho(a),\rho(a^*)]=0\]
and define $\PS_{\gamma}\subset \Mst_{\gamma}(\ol{Q})$ likewise, i.e. $\PS_{\gamma}=\PV_{\gamma}/\Gl_{\gamma}\cong\Mst_{\gamma}(\Pi_Q)$.

Corollary \ref{DDRstack} gives an isomorphism 
\[
\HO_c\left(\Mst_\gamma(\wt{Q}),\phi^{\mon}_{\TTTr(\wt{W}+W')}\QQ_{\Mst_\gamma(\wt{Q})}\right)\otimes\LLL^{-\gamma\cdot\gamma}\cong\HO_c\left(\PS_{\gamma}, \phim{\TTTr(W')}\mathbb{Q}_{\PS_{\gamma}}\right).
\]
Fix a cohomological degree $m$ and a number $N\gg 0$ depending on $m$.  Let $n=\sum_{i\in Q_0} \gamma_i$.  There is a natural embedding $\Gl_{\gamma}\hookrightarrow \Gl_n$ as a Levi subgroup.  Let $\Fr(n,N)$ be the space of $n$-tuples of linearly independent vectors in $\CC^N$.  
Let 
\begin{align*}
    \mathscr{M}\colonequals  &\XX_{\gamma}(\wt{Q})\times_{\Gl_{\gamma}} \Fr(n,N)\\
    \mathscr{M}'\colonequals  &\XX_{\gamma}(\ol{Q})\times_{\Gl_{\gamma}} \Fr(n,N)\\
    \mathscr{P}\colonequals  &\PV_{\gamma}\times_{\Gl_{\gamma}} \Fr(n,N)
\end{align*} 
and let $h_N\in\Gamma(\mathscr{M}')$ be the function induced by $\Tr(W')$.  We have natural maps $\pi$, $q_N$, $p_N$, and $\varpi$ fitting into the commutative diagram
\[
\begin{tikzcd}
\mathscr{M}\arrow{d}{q_N} \arrow{r}{\pi_N}& \mathscr{M}'\arrow{d}{p_N}
\\
\mathcal{M}_n(\wt{Q})\arrow{r}{\varpi}& 
\mathcal{M}_n(\ol{Q})
\end{tikzcd}
\]
where we define 
\[
\pi_N\colon (A,B,C,(v_1,\ldots,v_n))\mapsto (B,C,(v_1\ldots,v_n)).
\]
Set 
\[
\HO^m\colonequals \HO^i_c\left(\PS_{\gamma}, \phim{\TTTr(W')}\mathbb{Q}_{\PS_\gamma}\otimes\LLL^{\chi_Q(\gamma,\gamma)}\right).
\]
Below, for $\mathcal{F}\in\Ob(\Dub(\MMHM(\Msp(\ol{Q}))))$ we denote by $\mathcal{F}_{\gamma}$ the restriction of $\mathcal{F}$ to $\Msp_{\gamma}(\ol{Q})$.  Then there are isomorphisms
\begin{align*}
\HO&\substack{(0)\\ \cong}\HO^m_c\left( \mathscr{P},\phi_{h_N}^{\mon}\QQ_{\mathscr{P}}\otimes\LLL^{\chi_Q(\gamma,\gamma)-nN} \right)\\
&\cong\HO^m_c\left( \mathscr{M}',\phi^{\mon}_{h_N}\QQ_{\mathscr{P}}\otimes\LLL^{\chi_Q(\gamma,\gamma)-nN} \right)\\
&\substack{(1)\\\cong}\HO^m_c\left( \Msp_\gamma(\ol{Q}),\phi^{\mon}_{\TTr(W')}p_{N!}\QQ_{\mathscr{P}}\otimes\LLL^{\chi_Q(\gamma,\gamma)-nN} \right)\\
&\substack{(2)\\\cong}\HO^m_c\left( \Msp_\gamma(\ol{Q}),\phi^{\mon}_{\TTr(W')}p_{N!}\pi_{N!}\phi^{\mon}_{\TTr(\wt{W})}\QQ_{\mathscr{M}}\otimes\LLL^{\chi_{\wt{Q}}(\gamma,\gamma)/2-nN} \right)\\
&\cong\HO^m_c\left( \Msp_\gamma(\ol{Q}),\phi^{\mon}_{\TTr(W')}\varpi_!q_{N!}\phi^{\mon}_{\TTr(\wt{W})}\QQ_{\mathscr{M}}\otimes\LLL^{\chi_{\wt{Q}}(\gamma,\gamma)/2-nN} \right)\\
&\substack{(3)\\\cong}\HO^m_c\left( \Msp_\gamma(\ol{Q}),\phi^{\mon}_{\TTr(W')}\varpi_!\Sym\left( \BPS_{\wt{Q},\wt{W},\gamma}\otimes \HO_c(\pt/\CC^*)_{\vir}\right)_{\gamma}\right)\\
&\cong\HO^m_c\left( \Msp_\gamma(\ol{Q}),\phi^{\mon}_{\TTr(W')}\Sym\left( \varpi_!\BPS_{\wt{Q},\wt{W},\gamma}\otimes \HO_c(\pt/\CC^*)_{\vir}\right)_{\gamma}\right)\\
&\cong\HO^m_c\left( \Msp_\gamma(\ol{Q}),\phi^{\mon}_{\TTr(W')}\Sym\left( \BPS_{\Pi_Q,\gamma}\otimes\HO_c(\pt/\CC^*)\otimes\LLL\right)_{\gamma}\right)\\
&\substack{(4)\\\cong}\HO^m_c\left( \Msp_\gamma(\overline{Q}),\Sym\left( \mathcal{G}_{\gamma}\otimes\HO_c(\pt/\CC^*)\right)_{\gamma}\right)
\end{align*}
as required.  Isomorphism (0) follows as in \cite[Sec.2.2]{DaMe15b} from the fact that, up to removing a very high codimension substack, $\mathscr{P}$ is a $nN$-dimensional affine fibration over $\PS_{\gamma}$.  Isomorphism (1) follows from the fact that $p$ is approximated by proper maps (and so commutes with vanishing cycle functors \cite[Prop.4.3]{DaMe15b}).  Isomorphism (2) follows from usual cohomological dimensional reduction \cite[Thm.A.1]{DAV}.  Isomorphism (3) is the cohomological integrality theorem, while isomorphism (4) follows from commutativity of vanishing cycle functors with $\Sym$ \cite[Prop.3.11]{DaMe15b}.  

\end{proofof}

\subsection{Generalizing the CMPS conjecture}
With more effort, one can show that there are isomorphisms
\begin{equation}
    \label{DDRBPS}
\BPSA{\wt{Q},\wt{W}+W',\gamma}\cong \HO_c(\Msp_{\gamma}(\ol{Q}),\phim{\TTr(W')}\BPS_{\Pi_Q,\gamma})\otimes\LL^{1/2}.
\end{equation}
This follows from realising the deformed dimensional reduction isomorphism as an isomorphism of cohomological Hall algebras and realising BPS invariants as primitive generators.  It follows that we can endow 
\[
\bigoplus_{\gamma\neq 0}\HO(\Msp_{\gamma}(\ol{Q}),\phim{\Tr(W')}\BPS_{\Pi_Q,\gamma})
\]
with a Lie algebra structure as in \cite{DaMe15b}.  Expanding upon all this would greatly lengthen the paper.  However we do not need to prove that there is an isomorphism \eqref{DDRBPS} to obtain the following corollaries, generalizing the CMPS conjecture.

\begin{corollary}
\label{GenCor}
Let $\alpha\in\CC\langle b,c\rangle_{\geq 2}$ be quasihomogeneous, let $\eta\in\CC[b,c]$ be the Abelianization of $\alpha$, and assume that $\HO_c(\AAA{2},\phim{\eta}\QQ_{\AAA{2}})$ is pure.  The cohomological DT invariants for $(\LQ{3},\wt{W}+\alpha)$ and dimension $n\geq 1$ are
\[
\BPSA{\LQ{3},\wt{W}+\alpha,n}\cong \HO_c(\AAA{2},\phim{\eta}\QQ_{\AAA{2}})\otimes\LLL^{-1/2}
\]
and thus only depend on the Abelianization of $\alpha$ and do not depend on $n$ at all.  In particular, the cohomological DT invariants for $(\LQ{3},W_d)$ and dimension vector $n\geq 1$ are
\[
\BPSA{\LQ{3},W_d,n}\cong \HO_c(\AAA{1},\phim{c^d}\QQ_{\AAA{1}})\otimes\LLL^{1/2}.
\]

\end{corollary}

\begin{proof}
We define
\begin{align*}
\Delta_n\colon &\AAA{2}\rightarrow\Msp_n(\ol{Q})\\
&(y,z)\mapsto (y\cdot \Id_{n\times n},z\cdot \Id_{n\times n})
\end{align*}
Then by \cite[Sec.5]{Dav16b}, for all $n\geq 1$ we can write $\BPS_{\Pi_Q,n}=\Delta_{n*}\QQ_{\AAA{2}}\otimes\LLL^{-1}$.  In the notation of Theorem \ref{DGT_compare}, we have
\[
\mathcal{G}_n=\HO(\AAA{2},\phim{\eta}\QQ_{\AAA{2}})\otimes\LLL
\]
for all $n\geq 1$, and we have an isomorphism
\begin{align}
\label{pr1}
\bigoplus_{n \in\NN}\HO_c(\StComm_n,\phim{\TTTr(\alpha)}\QQ_{\StComm_n})\cong \Sym\left( \bigoplus_{\gamma\in\NN\setminus \{0\}} \HO(\AAA{2},\phim{\eta}\QQ_{\AAA{2}})\otimes\LLL^{1/2}\otimes\HO_c(\pt/\CC^*)_{\vir}\right).
\end{align}
In particular, the left hand side of \eqref{pr1} is an element of the semisimple category of \textit{pure} monodromic mixed Hodge structures, so that the isomorphism $\mathrm{LHS}\cong\Sym(\mathcal{T})$ determines $\mathcal{T}$ up to isomorphism.  Corollary \ref{DDRstack} and \eqref{coh_int_thm} give us isomorphisms
\begin{align}\nonumber
\bigoplus_{n \in\NN}\HO_c(\StComm_n,\phim{\TTTr(\alpha)}\QQ_{\StComm_n})\cong
&\bigoplus_{n\in\NN}\HO_c(\Mst_n(\wt{Q}),\phim{\TTTr(\wt{W}+\alpha)}\QQ_{\Mst_{n}(\wt{Q})})\otimes\LL^{-n\cdot n}\\
\cong&\Sym\left(\bigoplus_{n\geq 1} \BPSA{\LQ{3},\wt{W}+\alpha,n}\otimes\HO_c(\pt/\CC^*)_{\vir}\right).
\end{align}
and the result follows from comparing with \eqref{pr1}.
\end{proof}

We deduce, as a special case of Corollary \ref{GenCor}, a cohomological refinement of \cite[Thm.3.2]{CMPS}:

\begin{corollary}
Set $(Q,W)=(\LQ{3},a[b,c]-cb^2)$.  Then
\[
\BPSA{Q,W,n}\cong \LLL^{1/2}
\]
for all $n\geq 1$.
\end{corollary}
\begin{proof}
By Corollary \ref{GenCor}, this reduces to showing that
\[
\HO_c(\AAA{2},\phim{cb^2}\QQ_{\AAA{2}})\cong\HO_c(\AAA{1},\QQ),
\]
which follows from the usual dimensional reduction isomorphism, defining $g_1=b^2$ in \eqref{set1}.
\end{proof}

\bibliographystyle{plain}


\textsc{\small B. Davison: School of Mathematics, University of Edinburgh, James Clerk Maxwell Building, The King's Buildings, Peter Guthrie Tait Road, EH9 3FD}\\
\textit{\small E-mail address:} \texttt{\small ben.davison@ed.ac.uk}\\

\textsc{\small T. P\u adurariu: Department of Mathematics, Massachusetts Institute of Techonology, 
182 Memorial Drive, Cambridge, MA 02139}\\
\textit{\small E-mail address:} \texttt{\small tpad@mit.edu}\\

\end{document}